\documentclass[reqno,english]{amsart}%
\usepackage{amsfonts,amsmath,latexsym,verbatim,amscd,mathrsfs,color,array}
\usepackage{amsmath,amssymb,amsthm,amsfonts,graphicx,color}
\usepackage{amssymb}
\usepackage{epstopdf}
\usepackage{cite}
\usepackage{graphicx}
\usepackage{amsmath}
\usepackage{amsfonts}%
\providecommand{\U}[1]{\protect\rule{.1in}{.1in}}

\numberwithin{equation}{section}

\newtheorem{theorem}{Theorem}[section]
\newtheorem{corollary}{Corollary}[section]
\newtheorem{lemma}{Lemma}[section]
\newtheorem{proposition}{Proposition}[section]
\newtheorem{remark}{Remark}[section]
\newtheorem{definition}{Definition}[section]

\numberwithin{equation}{section}

\newcommand{\bbr}{\mathbb{R}}

\newcommand{\bbn}{\mathbb{N}}

\newcommand{\bd}{\begin{definition}}
\newcommand{\ed}{\end{definition}}

\newcommand{\br}{\begin{remark}}
\newcommand{\er}{\end{remark}}

\newcommand{\be}{\begin{equation}}
\newcommand{\ee}{\end{equation}}

\newcommand{\bc}{\begin{corollary}}
\newcommand{\ec}{\end{corollary}}

\begin{document}

\title[Coupled elliptic system]{ Ground states of Nonlinear Schr\"{o}dinger  System  with Mixed Couplings }

\author[J. Wei]{Juncheng Wei  }
\address{\noindent Department of Mathematics, University of British Columbia,
Vancouver, B.C., Canada, V6T 1Z2}
\email{jcwei@math.ubc.ca}

\author[Y.Wu]{Yuanze Wu}
\address{\noindent  School of Mathematics, China
University of Mining and Technology, Xuzhou, 221116, P.R. China }
\email{wuyz850306@cumt.edu.cn}

\begin{abstract}
We consider  the following $k$-coupled nonlinear Schr\"{o}dinger  system:
\begin{equation*}
\left\{\aligned&-\Delta u_j+\lambda_ju_j=\mu_ju_j^3+\sum_{i=1,i\not=j}^k\beta_{i,j} u_i^2u_j\quad\text{in }\bbr^N,\\
&u_j>0\quad\text{in }\bbr^N,\quad u_j(x)\to0\quad\text{as }|x|\to+\infty,\quad j=1,2,\cdots,k,\endaligned\right.
\end{equation*}
where $N\leq 3$, $k\geq3$, $\lambda_j,\mu_j>0$ are constants and $\beta_{i,j}=\beta_{j,i}\not=0$ are parameters. There have been intensive studies for the above system when $k=2$ or the system is purely  attractive ($ \beta_{i,j}>0, \forall i \not = j$) or purely repulsive ($\beta_{i,j}<0, \forall i\not = j $); however very few results are available for $k\geq 3$ when the system admits {\bf mixed couplings}, i.e., there exist $(i_1,j_1)$ and $(i_2,j_2)$ such that $\beta_{i_1,j_1}\beta_{i_2,j_2}<0$. In this paper we give the first systematic and an (almost) complete study on the existence of ground states when the system admits mixed couplings. We first  divide this system into {\bf repulsive-mixed} and {\bf total-mixed} cases. In the first case we prove nonexistence of ground states. In the second case we give an necessary condition for the existence of ground states and also provide  estimates for the Morse index.  The key idea is the {\bf block decomposition} of the system ({\bf optimal block decompositions, eventual block decompositions}), and the measure of total interaction forces  between different {\bf blocks}.  Finally the assumptions on the existence  of ground states are shown to be {\bf optimal} in some special cases.

\vspace{3mm} \noindent{\bf Keywords:} nonlinear Schr\"{o}dinger  system; ground state; mixed coupling; variational method; Morse index.

\vspace{3mm}\noindent {\bf AMS} Subject Classification 2010: 35B09; 35J47; 35J50.%

\end{abstract}

\date{}
\maketitle

\section{Introduction}
We consider the following $k$-coupled nonlinear Schr\"{o}dinger system
\begin{equation}\label{eqn0001}
\left\{\aligned&-\Delta u_j+\lambda_ju_j=\mu_ju_j^3+\sum_{i=1,i\not=j}^k\beta_{i,j} u_i^2u_j\quad\text{in }\bbr^N,\\
&u_j>0\quad\text{in }\bbr^N,\quad u_j(x)\to0\quad\text{as }|x|\to+\infty,\quad j=1,2,\cdots,k,\endaligned\right.
\end{equation}
where $N=1,2,3$, $k\geq3$, $\lambda_j,\mu_j>0$ are constants and $\beta_{i,j}=\beta_{j,i}\not=0$ are coupling parameters. (To simplify the notations, in the following,  we assume  $\beta_{j, j}=\mu_j$.) This paper is concerned with  the existence of  ground states in the general case $ k\geq 3$.

\medskip

It is well known that solutions of \eqref{eqn0001} are related to the solitary waves
of the Gross-Pitaevskii equations, which have applications in many physical models, such as in nonlinear optics and in Bose-Einstein condensates for multi-species condensates (cf. \cite{CLLL04,R03}).  Physically, in the system~\eqref{eqn0001}, $\mu_j$ and $\beta_{i,j}$ are the intraspecies and interspecies scattering lengths respectively, while $\lambda_j$ are from the chemical potentials.  The sign of the scattering length $\beta_{i,j}$ determines whether the interactions of states $i\rangle$ and $j\rangle$ are {\bf repulsive} ($\beta_{i,j}<0$) or {\bf attractive} ($\beta_{i,j}>0$).

\medskip

In the past fifteen years, the two-coupled case of the system~\eqref{eqn0001} (i.e. $k=2$) has been studied extensively in the literature.  An important feature of the two-coupled case is that it only has one coupling, i.e., $\beta_{12}=\beta_{21}$.  Thus, the two-coupled case of the system~\eqref{eqn0001} is either {\bf purely repulsive}  ($\beta_{12}=\beta_{21}<0$) or {\bf purely attractive} ($\beta_{12}=\beta_{21} >0$).  By using variational methods, Lyapunov-Schmidt reduction methods or bifurcation methods, various theorems, about the existence, multiplicity and  qualitative properties of nontrivial solutions of the two-coupled elliptic systems similar to \eqref{eqn0001}, have been established in the literature under various assumptions.  Since it seems almost impossible for us to provide a complete list of references, we refer the readers only to \cite{AC06,AC07,BDW10,BJS16,BS17,BS19,BTWW13,BW06,BWW07,CD10,CLLL04,CTV05,CZ13,DWW10,GJ16,GJ18,LW051,LW06,NTTV10,NTTV12,PW13,S07,ST15,
WW07,WW08,WW081,WWZ17,WZ19,W18} and the references therein.  Roughly speaking, in the two-coupled elliptic systems, the two components tend to segregate with each other in the repulsive case, which leads to phase separations and multi-existence of solutions, while the two components tend to synchronize with each other in the attractive case, which leads to uniqueness of the positive solution. For $k\geq 3$, the purely repulsive case and the purely attractive case of \eqref{eqn0001}, i.e., the couplings $\beta_{i,j}$ have the same sign for all $i\not=j$, have also been studied, see, for example, \cite{B13,LW08,LW10,SZ15,TT12,TT121,TV09,TW13,W17} and the references therein.

\medskip

However, a significant new feature of \eqref{eqn0001} for $k\geq3$ is the presence of  {\bf mixed  couplings}, i.e., there exist $(i_1,j_1)$ and $(i_2,j_2)$ such that $\beta_{i_1,j_1}\beta_{i_2,j_2}<0$. As far as we know, \eqref{eqn0001} for $k\geq3$ with  mixed couplings is less studied in the literature, and the only references are \cite{BSW16,BSW161,DW12,LW05,SW15,SW151,S15,ST16,STTZ16,
TTVW11}.  The primary goal of this paper is to give a complete study about the existence of ground states in the case of {\bf mixed couplings}.  In what follows, for the sake of clarity, let us first introduce some necessary notations and definitions.

\medskip

Let $\mathcal{H}_j$ be the Hilbert space of $H^1(\bbr^N)$ with the inner product
\begin{eqnarray*}\label{eqn0002}
\langle u,v\rangle_{\lambda_j}=\int_{\bbr^N}\nabla u\nabla v+\lambda_juvdx.
\end{eqnarray*}
Its corresponding norm is given by
\begin{eqnarray*}\label{eqn0003}
\|u\|_{\lambda_j}=\langle u,u\rangle_{\lambda_j}^{\frac12}.
\end{eqnarray*}
Let the energy functional of \eqref{eqn0001} be given by
\begin{eqnarray}
\mathcal{E}(\overrightarrow{u})=\frac{1}{2}\sum_{j=1}^k\|u_j\|_{\lambda_j}^2-\frac{1}{4}\sum_{j=1}^k\mu_j\|u_j\|_4^4-\frac{1}{2}\sum_{i,j=1, i<j}^k\beta_{i,j}\|u_iu_j\|_2^2,
\end{eqnarray}
where $\overrightarrow{u}=(u_1,u_2,\cdots,u_k)$ and $\|\cdot\|_p$ is the usual norm in $L^p(\bbr^N)$.  Then, $\mathcal{E}(\overrightarrow{u})$ is of class $C^2$ in $\mathcal{H}:=\prod_{j=1}^k\mathcal{H}_j$.  $\overrightarrow{v}$ is called a positive critical point of $\mathcal{E}(\overrightarrow{u})$ if $\mathcal{E}'(\overrightarrow{v})=\overrightarrow{0}$ in $\mathcal{H}^{-1}$ with $v_j>0$ for all $j$, where $\mathcal{H}^{-1}$ is the dual space of $\mathcal{H}$.  For $N\leq 3$, the standard elliptic regularity theory yields that
positive critical points of $\mathcal{E}(\overrightarrow{u})$ are equivalent to classical solutions of \eqref{eqn0001}.
We define the Nehari manifold of $\mathcal{E}(\overrightarrow{u})$ as follows:
\begin{eqnarray}
\label{Neh1}
\mathcal{N}=\{\overrightarrow{u}\in\widetilde{\mathcal{H}}\mid \overrightarrow{\mathcal{G}}(\overrightarrow{u})=(\mathcal{G}_1(\overrightarrow{u}), \mathcal{G}_2(\overrightarrow{u}),\cdots, \mathcal{G}_k(\overrightarrow{u}))=\overrightarrow{0}\},
\end{eqnarray}
where $\mathcal{G}_j(\overrightarrow{u})=\|u_j\|_{\lambda_j}^2-\mu_j\|u_j\|_4^4-\sum_{i=1,i\not=j}^k\beta_{i,j}\|u_iu_j\|_2^2$
and $\widetilde{\mathcal{H}}=\prod_{j=1}^k(\mathcal{H}_j\backslash\{0\})$.  Clearly, $\mathcal{N}$ contains all positive critical points of $\mathcal{E}(\overrightarrow{u})$.
Let
\begin{eqnarray}
\label{CN}
\mathcal{C}_{\mathcal{N}}=\inf_{\mathcal{N}}\mathcal{E}(\overrightarrow{u}).
\end{eqnarray}
Then, $\mathcal{C}_{\mathcal{N}}$ is well defined and nonnegative.  $\overrightarrow{v}$ is called a ground state of \eqref{eqn0001}, if $\overrightarrow{v}$ is a positive critical point of $\mathcal{E}(\overrightarrow{u})$ with $\mathcal{E}(\overrightarrow{v})=\mathcal{C}_{\mathcal{N}}$.

\medskip

We now continue our discussions on \eqref{eqn0001} for $k\geq3$ with the mixed couplings.  Most of the literature (cf. \cite{BSW16,BSW161,SW15,SW151,S15,ST16}) is devoted to the ``restricted'' ground states of \eqref{eqn0001} for $k\geq3$ with the mixed couplings, by either assuming that $u_j$ are all radially symmetric or considering \eqref{eqn0001} in a bounded domain $\Omega$.  The only paper, which is devoted to the ground states of \eqref{eqn0001}, is \cite{LW05}, where the existence and nonexistence of the ground states of \eqref{eqn0001} with mixed couplings were partially studied when $k=3$.  Thus, the existence of the ground states of \eqref{eqn0001}, for $k\geq3$ with the mixed couplings, remains largely open.   In this paper we give the first result on the existence and nonexistence of the ground states of \eqref{eqn0001} for $k\geq3$ with the mixed couplings, which can be summarized as follows (see Theorem \ref{thm0002} below):
\begin{enumerate}
\item[$(1)$]  Under some technical conditions, (which can be shown to be optimal in some special cases), \eqref{eqn0001} for $k\geq3$ has a ground state in the cases of the {\bf total-mixed couplings}  (the definition can be seen below);
\item[$(2)$]  \eqref{eqn0001} for $k\geq3$ has no ground states in the cases of the {\bf repulsive-mixed couplings} (the definition can also be seen below).
\end{enumerate}

\section{Block Decompositions and Statements of Main Results when $k=3, 4$}
Before we present the results in the general case $ k\geq 3$, we
first explain key  ideas, concepts and main results  when $k=3$ or $4$. We first consider the case $k=3$:
\begin{equation}\label{eqnew0001}
\left\{\aligned&-\Delta u_1+\lambda_1u_1=\mu_1u_1^3+\beta_{1,2} u_2^2u_1+\beta_{1,3}u_3^2u_1\quad\text{in }\bbr^N,\\
&-\Delta u_2+\lambda_2u_2=\mu_2u_2^3+\beta_{1,2} u_1^2u_2+\beta_{2,3}u_3^2u_2\quad\text{in }\bbr^N,\\
&-\Delta u_3+\lambda_3u_3=\mu_3u_3^3+\beta_{1,3} u_1^2u_3+\beta_{2,3}u_2^2u_3\quad\text{in }\bbr^N,\\
&u_i>0\quad\text{in }\bbr^N,\quad u_i(x)\to0\quad\text{as }|x|\to+\infty,\quad i=1,2,3.\endaligned\right.
\end{equation}
We start by recalling known results about \eqref{eqnew0001} in the literature.  As pointed out in \cite{LW05}, there are actually only {\bf four} cases of the couplings:
\begin{enumerate}
\item[$(a)$] The purely attractive case: $\beta_{1,2}>0$, $\beta_{1,3}>0$ and $\beta_{2,3}>0$;
\item[$(b)$] The purely repulsive case: $\beta_{1,2}<0$, $\beta_{1,3}<0$ and $\beta_{2,3}<0$;
\item[$(c)$] The mixed case~$(1)$: $\beta_{1,2}>0$, $\beta_{1,3}<0$ and $\beta_{2,3}<0$;
\item[$(d)$] The mixed case~$(2)$: $\beta_{1,2}>0$, $\beta_{1,3}>0$ and $\beta_{2,3}<0$.
\end{enumerate}

The first two cases (a) and (b) are reminiscent of the $k=2$ case, which can be dealt with similarly.
In the mixed case~$(c)$, the system~\eqref{eqnew0001} can be seen  as a coupled system between an attractively two-coupled system about $(u_1, u_2)$ and  a single equation about $u_3$.  Since $\beta_{1,3}<0$ and $\beta_{2,3}<0$, the interaction between the two-coupled system and the single equation is ``repulsive''.   We re-name this mixed case as the {\bf repulsive-mixed case}. Similar to the repulsive case of $k=2$ (cf. \cite{LW05}), the ground state of \eqref{eqnew0001} does not exist in this case (under some technical conditions). (However,  if  $u_j$ are all radially symmetric or one considers \eqref{eqnew0001} in a bounded domain $\Omega$, then the ``restricted'' ground states of \eqref{eqnew0001} exist for some ranges of $\beta_{i,j}$ (cf. \cite{BSW16,BSW161,SW15,SW151,S15,ST16}).)

\medskip

The most difficult (and interesting) case is  the mixed case~$(d)$. If we still regard the system~\eqref{eqnew0001} as an attractively two-coupled system coupled  with a single equation, then the situation is much more complicated than that in the repulsive-mixed case~$(c)$, since the coupling between them can be both repulsive ($\beta_{2,3}<0$) and attractive ($\beta_{1,2}>0$, $\beta_{1,3}>0$).  We re-name this mixed case as  the {\bf total-mixed case}. In the bounded domain with Dirichlet boundary condition, the existence of the ``restricted'' ground states of \eqref{eqnew0001}, in the total-mixed case~$(d)$, has been studied in \cite{S15,ST16} for some ranges of $\beta_{i,j}$.  However,  it has been proved in \cite{LW05}, by using Lyapunov-Schmidt reduction methods, that \eqref{eqnew0001} has a non-radially symmetric solution in the total-mixed case~$(d)$ for $|\beta_{i,j}|$ all sufficiently small and $|\beta_{2,3}|>>|\beta_{1,2}|,|\beta_{1,3}|$.  Moreover, the energy value of this non-radially symmetric solution is strictly less than that of the uniquely radially symmetric solution of \eqref{eqnew0001} for $|\beta_{i,j}|$ all sufficiently small.  This result suggests that the ground states of \eqref{eqnew0001}, if they exist, are non-radially symmetric in the total-mixed case~$(d)$, at least for $|\beta_{i,j}|$ all sufficiently small and $|\beta_{2,3}|>>|\beta_{1,2}|,|\beta_{1,3}|$.  By our above discussions, in the total-mixed case~$(d)$, the major task, in studying the existence of the ground states of \eqref{eqnew0001}, is to measure the {\bf total interaction} between the attractively two-coupled system and the single equation, near the least energy value $\mathcal{C}_{\mathcal{N}}$.  It turns out in this case the total interaction  can mainly be controlled by the linear term $\lambda_j$.

\medskip

The following theorem gives complete characterization of the existence and nonexistence of ground states of \eqref{eqnew0001}.

\begin{theorem}\label{coro0001}
Let $N=1,2,3$.
\begin{enumerate}
\item[$(1)$]  In the purely attractive case~$(a)$, there exist $0<\beta_0<\widehat{\beta}_0$ such that
    \begin{enumerate}
    \item[$(i)$] \eqref{eqnew0001} has a ground state with the Morse index 3 for $0<\beta_{1,2},\beta_{1,3},\beta_{2,3}<\beta_0$;
    \item[$(ii)$]  \eqref{eqnew0001} has a ground state with the Morse index 2 for $\beta_{1,2}>\widehat{\beta}_0$, $0<\beta_{1,3},\beta_{2,3}<\beta_0$.
    \item[$(iii)$]  \eqref{eqnew0001} has a ground state with the Morse index 1 for $\beta_{i,j}>\widehat{\beta}_0$ and $|\beta_{i,j}-\beta_{i,l}|<<1$ with all $i,j,l=1,2,3$, $i\not=j$, $i\not=l$ and $j\not=l$, provided that $|\lambda_i-\lambda_j|<<1$ for all $i,j=1,2,3$ with $i\not=j$.
    \end{enumerate}

\item[$(2)$]  In the purely repulsive case~$(b)$ and in the repulsive-mixed case~$(c)$, $\mathcal{C}_{\mathcal{N}}$ can not be attained, provided that the coefficient matrix $\Theta= (\beta_{i,j})$ is positively definite.   That is, system~\eqref{eqnew0001} have no ground states.

\item[$(3)$] In the total-mixed case~$(d)$, if $\lambda_1<\min\{\lambda_2,\lambda_3\}$, then there exist $0<\beta_0<\widehat{\beta}_0$ such that
    \begin{enumerate}
    \item[$(i)$] \eqref{eqnew0001} has a ground state with the Morse index 3 for $0<\beta_{1,2}<\beta_0$, $0<\beta_{1,3}<\beta_0$ and $\beta_{2,3}<0$;
    \item[$(ii)$]  \eqref{eqnew0001} has a ground state with the Morse index 2 for $\beta_{1,2}>\widehat{\beta}_0$, $0<\beta_{1,3}<\beta_0$ and $\beta_{2,3}<0$.
    \end{enumerate}

\item[$(4)$] In the total-mixed case~$(d)$,   let $\beta_{1,2}=\delta\widehat{\beta}_{1,2}$, $\beta_{1,3}=\delta^t\widehat{\beta}_{1,3}$ and $\beta_{2,3}=-\delta^s\widehat{\beta}_{2,3}$, where $\delta>0$ is a parameter and $t,s,\widehat{\beta}_{i,j}$ are absolutely positive constants.  If $\lambda_1\geq\min\{\lambda_2,\lambda_3\}$ and $0<s<\min\{1,t\}$, then for $\delta$ sufficiently small, $\mathcal{C}_{\mathcal{N}}$ can not be attained. That is, \eqref{eqnew0001} has no ground states.

\end{enumerate}
\end{theorem}

\begin{remark}\label{rmk0003}
\begin{enumerate}

\item[$(a)$] (4) of Theorem \ref{coro0001} shows that the ground state in Corollary 1 of \cite{LW05} does not exist.
\item[$(b)$]  As we pointed out above, the major difficulty in proving the existence part of Theorem~\ref{coro0001} is to measure the interaction terms
\begin{eqnarray}
\beta_{1,3}\|u_1u_3\|_2^2+\beta_{2,3}\|u_2u_3\|_2^2\quad\text{and}\quad\beta_{1,2}\|u_1u_2\|_2^2+\beta_{2,3}\|u_2u_3\|_2^2
\end{eqnarray}
by non-radially symmetric vector-functions.
By using the ground states of the system of $(u_1, u_2)$ and the single equation of $u_3$ (or the pair of $(u_1, u_3)$ and $u_2$) as test functions we find that the above interaction terms behave like:
\begin{eqnarray*}
\mathfrak{H}&=&\sup_{R>>1}(C\beta_{1,3}R^{1-N+\gamma}e^{-2\min\{\sqrt{\lambda_1}, \sqrt{\lambda_3}\}R}\\
&&+C'\beta_{2,3}R^{1-N+\gamma'}e^{-2\min\{\sqrt{\lambda_2}, \sqrt{\lambda_3}\}R})
\end{eqnarray*}
and
\begin{eqnarray*}
\mathfrak{G}&=&\sup_{R>>1}(C\beta_{1,2}R^{1-N+\gamma''}e^{-2\min\{\sqrt{\lambda_1}, \sqrt{\lambda_2}\}R}\\
&&+C'\beta_{2,3}R^{1-N+\gamma'}e^{-2\min\{\sqrt{\lambda_2}, \sqrt{\lambda_3}\}R}),
\end{eqnarray*}
where $\gamma,\gamma',\gamma''$ are positive constants depending only on $N$ and the relation of $\lambda_j$, and $C,C'$ are positive constants (depending on the ground states of small system $(u_1, u_2)$ or $(u_1, u_3)$).  Moreover, roughly speaking, if $\min\{\mathfrak{H}, \mathfrak{G}\}>0$ then the interaction between the system of $(u_1, u_2)$ and the single equation of $u_3$ (or the pair of $(u_1, u_3)$ and $u_2$) is ``attractive'' and consequently the ground states exist; while if $\min\{\mathfrak{H}, \mathfrak{G}\}<0$ then the interaction between the system of $(u_1, u_2)$ and the single equation of $u_3$ (or the pair of $(u_1, u_3)$ and $u_2$) is ``repulsive'' and consequently the ground states do not exist.  Based on this observation, if we further assume that $0<-\beta_{2,3}<<\min\{\beta_{1,2}, \beta_{1,3}\}$ in the case $\lambda_1=\min\{\lambda_2,\lambda_3\}$, then the ground states of \eqref{eqnew0001} still exists.  Thus, Theorem \ref{coro0001} gives an almost complete result about the existence and nonexistence of ground states of \eqref{eqnew0001}.
\item[$(c)$] The existence of the ground states of \eqref{eqnew0001} with the Morse index $3$ for the purely attractive case and the nonexistence of the ground states of \eqref{eqnew0001} for the repulsive-mixed case is actually proved in
    \cite[Corollary~1.3 and Theorem~1.6]{ST16}, respectively.  We list them in Theorem~\ref{coro0001} for the sake of completeness.  The existence of ground states of \eqref{eqnew0001} with the Morse index $1$ for the purely attractive case is proved in \cite[Theorem~2.1]{LW10}.  Here, our provide a different proof of this result.
\end{enumerate}
\end{remark}

\medskip

As we stated above, in proving Theorem~\ref{coro0001}, our major idea is to regard the three-coupled system~\eqref{eqnew0001} as an attractively two-coupled system  coupled with a single equation, and to precisely measure the interaction between them.  To extend the  above idea to the general $k$-coupled system~\eqref{eqn0001} for $k\geq 4$, we need to further decompose the $k$-coupled system~\eqref{eqn0001}, which is based on the following concepts of {\bf optimal block decomposition} and {\bf  eventual block decomposition}.  These definitions for the general $k$-component cases are tedious and lengthy, which we would like to state at the next section and only introduce the key steps here:  first we group all attractive components $\mu_j$ together into blocks of sub-matrices so that inside each block the interactions between components are all attractive.  The decomposition is called {\bf optimal} if the number of blocks needed is the least, and the number of the blocks is called the {\bf degree} of this {\bf optimal block decomposition} and is denoted by $d$.  In the second step we need to group different "attractive" blocks together to form larger blocks. To see if two blocks are attractive or repulsive, we need to define quantities, named {\bf interaction forces}, which measure the interaction between different blocks in an optimal block decomposition.  Roughly speaking, if the quantity is positive then the interaction between of corresponding blocks is ``attractive'', while if this quantity is negative then the interaction between of these two blocks is ``repulsive''.  We now  group all possible ``attractive'' blocks together into bigger blocks of sub-matrices so that inside each bigger block the forces between blocks are all ``attractive''.  We repeat these steps  until we can not group them in this way anymore.  Then the remaining  matrix, consisting of ``largest'' attractive blocks, is called an {\bf eventual block decomposition}, and the number of the ``largest'' blocks is called the {\bf  degree} of an eventual block  decomposition and is denoted by $m$. More precise definitions can be found at the next section. Let us test these ideas with the first nontrivial case $k=4$:

\begin{equation}\label{eqnewnew0001}
\left\{\aligned&-\Delta u_1+\lambda_1u_1=\mu_1u_1^3+\beta_{1,2} u_2^2u_1+\beta_{1,3}u_3^2u_1+\beta_{1,4}u_4^2u_1\quad\text{in }\bbr^N,\\
&-\Delta u_2+\lambda_2u_2=\mu_2u_2^3+\beta_{1,2} u_1^2u_2+\beta_{2,3}u_3^2u_2+\beta_{2,4}u_4^2u_2\quad\text{in }\bbr^N,\\
&-\Delta u_3+\lambda_3u_3=\mu_3u_3^3+\beta_{1,3} u_1^2u_3+\beta_{2,3}u_2^2u_3+\beta_{3,4}u_4^2u_3\quad\text{in }\bbr^N,\\
&-\Delta u_4+\lambda_4u_4=\mu_4u_4^3+\beta_{1,4} u_1^2u_4+\beta_{2,4}u_2^2u_4+\beta_{3,4}u_3^2u_4\quad\text{in }\bbr^N,\\
&u_i>0\quad\text{in }\bbr^N,\quad u_i(x)\to0\quad\text{as }|x|\to+\infty,\quad i=1,2,3,4.\endaligned\right.
\end{equation}
We assume that the coefficients satisfy
\begin{equation}
\label{Hdef}
(H)\ \ \ \ \ \ \ \beta_{1,2}>0, \beta_{1,3}>0, \beta_{1,4}<0, \beta_{2,3}<0, \beta_{2,4}>0, \beta_{3,4}<0.
\end{equation}
  Clearly, an optimal block decomposition in this case can be given by
\begin{eqnarray}\label{eqnewnew0007}
\label{A1}
\mathbf{A}_1=\left(\aligned \left(\aligned &\mu_1\quad \beta_{1,2}\\
&\beta_{1,2}\quad \mu_2 \endaligned\right)
\ \ \left(\aligned&\beta_{1,3}\\
&\beta_{2,3}\endaligned\right)\ \ \left(\aligned&\beta_{1,4}\\
&\beta_{2,4}\endaligned\right)\\
\left(\beta_{1,3}\quad \beta_{2,3}\right)\quad\quad\mu_3 \quad\quad\quad\beta_{3,4}\\
\left(\beta_{1,4}\quad \beta_{2,4}\right)
\quad\quad\beta_{3,4} \quad \quad\quad\mu_4\endaligned\right)
\end{eqnarray}
with  degree $d=3$.  To obtain eventual block  decomposition, we need to first define the interaction forces.  To do this, we rewrite $\mathbf{A}_1$ as follows:
\begin{eqnarray}
\mathbf{A}_1=\left(\aligned B_{1,1}\quad B_{1,2}\quad B_{1,3}\\
B_{1,2}\quad B_{2,2}\quad B_{2,3}\\
B_{1,3}\quad B_{2,3}\quad B_{3,3}\endaligned\right).
\end{eqnarray}
Here $B_{i,j}$ are given in (\ref{A1}). For example $B_{2,2}=\mu_3, B_{3,3}= \mu_4$.
Since the ground states in $B_{1,1}$ exist for some ranges of $\beta_{1,2}$ and the ground states in $B_{2,2}$ and $B_{3,3}$ also exist, and they  all have  exponentially decaying at infinity, we may define quantities
\begin{eqnarray}\label{eqnewnew0002}
\mathfrak{F}_{1,2}^0=\sup_{R>>1}&(&C_{1,3}^{1,2}\beta_{1,3}R^{1-N+\gamma_{1,3}}e^{-2\min\{\sqrt{\lambda_1}, \sqrt{\lambda_3}\}R}\notag\\
&&+C_{2,3}^{1,2}\beta_{2,3}R^{1-N+\gamma_{2,3}}e^{-2\min\{\sqrt{\lambda_2}, \sqrt{\lambda_3}\}R}),
\end{eqnarray}
\begin{eqnarray*}
\mathfrak{F}_{1,3}^0=\sup_{R>>1}&(&C_{1,4}^{1,3}\beta_{1,4}R^{1-N+\gamma_{1,4}}e^{-2\min\{\sqrt{\lambda_1}, \sqrt{\lambda_4}\}R}\notag\\
&&+C_{2,4}^{1,3}\beta_{2,4}R^{1-N+\gamma_{2,4}}e^{-2\min\{\sqrt{\lambda_2}, \sqrt{\lambda_4}\}R})
\end{eqnarray*}
and
\begin{eqnarray*}
\mathfrak{F}_{2,3}^0=\sup_{R>>1}(C_{3,4}^{2,3}\beta_{3,4}R^{1-N+\gamma_{3,4}}e^{-2\min\{\sqrt{\lambda_3}, \sqrt{\lambda_4}\}R}),
\end{eqnarray*}
where $\gamma_{i,j}$ are positive constants depending only on $N$ and the relation of $\lambda_j$, and $C_{i,j}^{s,t}$ are positive constants depending only on the ground states in the corresponding blocks.  Since the ground states in blocks with the same least critical value is compact, $C_{i,j}^{s,t}$ is uniformly bounded from below and above.  These quantities $\mathfrak{F}_{i,j}^0$, as $\mathfrak{H}$ and $\mathfrak{G}$, are used to measure the interaction between the blocks $B_{i,i}$ and $B_{j,j}$ from the viewpoint of the concentration-compactness principle. Roughly speaking, the sign of $\mathfrak{F}_{i,j}^0$ determines whether the blocks $B_{i,i}$ and $B_{j,j}$ are ``attractive'' $(\mathfrak{F}_{i,j}^0>0)$ or ``repulsive'' $(\mathfrak{F}_{i,j}^0<0)$.  Note that  $\mathfrak{F}^0_{2,3}<0$. If both $\mathfrak{F}^0_{1,2}<0$ and $\mathfrak{F}^0_{1,3}<0$, then the blocks in $\mathbf{A}_1$ can not be further grouped into ``bigger'' blocks so that inside each bigger block the interaction forces between blocks are all ``attractive''.  Thus, $\mathbf{A}_1$ is also an eventual block decomposition with degree $m=3$.  If either $\mathfrak{F}_{1,2}^0>0$ or $\mathfrak{F}_{1,3}^0>0$, then roughly speaking, by Theorem~\ref{coro0001} there exists a ground state in the ``bigger'' block:
\begin{eqnarray*}
C_{1,1}=\left(\aligned B_{1,1}\quad B_{1,2}\\
B_{1,2}\quad B_{2,2}\endaligned\right).
\end{eqnarray*}
Here, without loss of generality, we assume $\mathfrak{F}_{1,2}^0>0$ (the other case $\mathfrak{F}_{1,3}^0>0$ is similar).
Thus, we may further group $\mathbf{A}_1$ as follows:
\begin{eqnarray*}
\mathbf{A}_2=\left(\aligned &\left(\aligned B_{1,1}\quad B_{1,2}\\
B_{1,2}\quad B_{2,2}\endaligned\right)\quad\left(\aligned B_{1,3}\\ B_{2,3}\endaligned\right)\\
&\left(\aligned B_{1,3}\quad B_{2,3}\endaligned\right)\quad\quad B_{3,3}\endaligned\right).
\end{eqnarray*}
We rewrite $\mathbf{A}_2$ by
\begin{eqnarray}\label{eqnewnew0008}
\mathbf{A}_2=\left(\aligned C_{1,1}\quad C_{1,2}\\
C_{1,2}\quad C_{2,2}\endaligned\right)
\end{eqnarray}
and define the interaction force between $C_{1,1}$ and $C_{2,2}$ by $\mathfrak{F}^1_{1,2}=\mathfrak{F}^0_{1,3}+\mathfrak{F}^0_{2,3}$, which as $\mathfrak{F}_{i,j}^0$, is used to measure the interaction between the blocks $C_{1,1}$ and $C_{2,2}$, and roughly speaking,
the sign of $\mathfrak{F}_{1,2}^1$ determines whether the blocks $C_{1,1}$ and $C_{2,2}$ are ``attractive'' $(\mathfrak{F}_{1,2}^1>0)$ or ``repulsive'' $(\mathfrak{F}_{1,2}^1<0)$.  If $\mathfrak{F}^1_{1,2}<0$ then the blocks in $\mathbf{A}_2$ can not be further grouped into ``bigger'' blocks so that inside each bigger block the interaction forces between blocks are all ``attractive''.  Thus, $\mathbf{A}_2$ is an eventual block decomposition with degree $m=2$.  If $\mathfrak{F}^1_{1,2}>0$ then we may further group $\mathbf{A}_2$ as a whole element
\begin{eqnarray*}
\mathbf{A}_3=\left(\aligned \left[\aligned C_{1,1}\quad C_{1,2}\\
C_{1,2}\quad C_{2,2}\endaligned\right]\endaligned\right).
\end{eqnarray*}
Since $\mathbf{A}_3$ only has one block, we can not further group it into a ``bigger'' block.  Therefore, $\mathbf{A}_3$ is an eventual block  decomposition with degree $m=1$.  There are another optimal block decomposition with the blocks $(u_1,u_3)$, $u_2$ and $u_4$.  One can use the same method to obtain its eventual block  decompositions and count their degrees.  Since the defined interaction forces almost determine whether the corresponding blocks are ``attractive'' or ``repulsive'', roughly speaking, the degrees of eventual block decompositions determine the number of groups of the components $u_j$ that ``stay together''.  Therefore, the ground states of \eqref{eqnewnew0001} are expected to exist if and only if the degrees of all eventual block  decompositions equal to $1$.  Now, our results for \eqref{eqnewnew0001} in the case $(H)$ can be stated as follows.
\begin{theorem}\label{thm0003}
Let $N=1,2,3$.  Then in the case $(H)$ at (\ref{Hdef}),
\begin{enumerate}
\item[$(1)$]  if $\lambda_1=\lambda_2<\min\{\lambda_3,\lambda_4\}$ and $0<-\beta_{2,3},-\beta_{1,4},-\beta_{3,4}<<\beta_{1,2},\beta_{2,4},\beta_{1,3}$ then there exist $\widehat{\beta}_0>\beta_0>0$ such that
    \begin{enumerate}
    \item[$(i)$]  if $\beta_{1,2},\beta_{1,3}<\beta_0$ then \eqref{eqnewnew0001} has a ground state with the Morse index $4$,
    \item[$(ii)$]  if $\beta_{1,3}<\beta_0$ and $\beta_{1,2}>\widehat{\beta}_0$ then \eqref{eqnewnew0001} has a ground state with the Morse index $3$.
    \end{enumerate}
\item[$(2)$]  Assume $\beta_{1,2}=\delta^{t_{1,2}}\widehat{\beta}_{1,2}$, $\beta_{1,3}=\delta^{t_{1,3}}\widehat{\beta}_{1,3}$, $\beta_{2,3}=-\delta^{t_{2,3}}\widehat{\beta}_{2,3}$, $\beta_{1,4}=-\delta^{t_{1,4}}\widehat{\beta}_{1,4}$, $\beta_{2,4}=\delta^{t_{2,4}}\widehat{\beta}_{2,4}$ and $\beta_{3,4}=-\delta^{t_{3,4}}\widehat{\beta}_{3,4}$, where $t_{i,j}$ and $\widehat{\beta}_{i,j}$ are all positively absolute constants and $\delta>0$ is a small parameter.  If $\min\{\lambda_3,\lambda_4\}<\min\{\lambda_1,\lambda_2\}$ and $\max\{t_{2,3},t_{1,4}, t_{3,4}\}<t_{1,2}<\min\{t_{1,3}, t_{2,4}\}$, then \eqref{eqnewnew0001} has no ground states for $\delta>0$ sufficiently small.
\end{enumerate}
\end{theorem}

\begin{remark}
As in Theorem~\ref{coro0001}, the assumptions $\lambda_1=\lambda_2<\min\{\lambda_3,\lambda_4\}$ and $0<-\beta_{2,3},-\beta_{1,4},-\beta_{3,4}<<\beta_{1,2},\beta_{2,4},\beta_{1,3}$ are used to grantee all eventual block  decompositions have the degree $m=1$, and it can be slightly generalized as that in $(b)$ of Remark~\ref{rmk0003}.
\end{remark}

\medskip

For other cases of the couplings of the four-coupled system~\eqref{eqnewnew0001} or for the general $k$-coupled system~\eqref{eqn0001}, the strategy is the same.  However, to state our results for the general $k$-coupled system~\eqref{eqn0001}, we need to rigorously define optimal block  decompositions and eventual block   decompositions.

\section{Block Decompositions and Statements of Main Results in the General Case}

Let us first define optimal block decompositions.  Let $d=1,2,\cdots,k$, $0=a_0<a_1<\cdots<a_{d-1}<a_d=k$ and
\begin{eqnarray}\label{eqn0098}
\mathcal{K}_{t,s,\mathbf{a}_d}=(a_{t-1}, a_{t}]_{\bbn}\times(a_{s-1}, a_{s}]_{\bbn},
\end{eqnarray}
where $\mathbf{a}_d=(a_0,a_1,\cdots,a_d)$, $t,s=1,2,\cdots,d$ and $(a_{t-1}, a_{t}]_{\bbn}=(a_{t-1}, a_{t}]\cap\bbn$.  Then,
\begin{eqnarray*}
\mathbf{A}_d=([\beta_{i,j}]_{(i,j)\in\mathcal{K}_{t,s,\mathbf{a}_d}})_{t,s=1,2,\cdots,d}
\end{eqnarray*}
is called a {\bf $d$-decomposition} of the coefficient matrix $\Theta=(\beta_{i,j})$.  Moreover, $\mathbf{A}_d$ is called repulsive if the couplings $\beta_{i,j}$ are all negative, $\mathbf{A}_d$ is called attractive if the couplings $\beta_{i,j}$ are all positive and $\mathbf{A}_d$ is called mixed if the couplings $\beta_{i,j}$ are mixed.  In $\mathbf{A}_d$, $\Theta_{t,s}=[\beta_{i,j}]_{(i,j)\in\mathcal{K}_{t,s,\mathbf{a}_d}}$ is called the $(t,s)$ block of $\mathbf{A}_d$.  Moreover, if $\{i\not=j,(i,j)\in\mathcal{K}_{s,s,\mathbf{a}_d}\}\not=\emptyset$, then all couplings~$\beta_{i,j}$ with $i\not=j$ in the $(s,s)$ block $\Theta_{s,s}$ are called the $s_{th}$ inner-couplings, while the couplings~$\beta_{i,j}$ in all $(s,t)$ blocks $\Theta_{s,t}$ with $s\not=t$ are called the inter-couplings.

\vskip0.12in

Let $\mathbf{i}=({i_1}, {i_2}, \cdots, {i_k})$ be a permutation of $(1,2,\cdots,k)$.  Then, correspondingly
$$
\Theta_{\mathbf{i}}=[\beta_{i_j,i_l}]_{j,l=1,2,\cdots,k}
$$
is a permutation of $\Theta$.
For the sake of clarity, we denote the corresponding $d$-decomposition of $\Theta_{\mathbf{i}}$ by $\mathbf{A}_{d,\mathbf{i}}$.
For the mixed couplings, there exist $\mathbf{i}=({i_1}, {i_2}, \cdots, {i_k})$, a permutation of $(1,2,\cdots,k)$, and $d=2,3\cdots,k-1$ such that $\Theta_{\mathbf{i}}$ has a mixed $d$-decomposition $\mathbf{A}_{d,\mathbf{i}}$ with all inner-couplings being positive.
Let $\mathbf{A}_{d,\mathbf{i}}$ be a mixed $d$-decomposition of $\Theta_{\mathbf{i}}$ such that all inner-couplings are positive.
$\mathbf{A}_{d,\mathbf{i}}$ is called an {\bf optimally mixed block decomposition of $\Theta$ to the permutation $\mathbf{i}$}, if for any $n<d$ and any $n$-decomposition of $\Theta_{\mathbf{i}}$, there exists at least one negative inner-coupling.
By our definitions, an optimally mixed block decomposition of $\Theta$ to the permutation $\mathbf{i}$, say $\mathbf{A}_{d,\mathbf{i}}$, is the one that, the number of the $(s,s)$ blocks of $\mathbf{A}_{d,\mathbf{i}}$ is the smallest in all decompositions of $\Theta_{\mathbf{i}}$, whose inner-couplings are all positive.  Clearly, for a given permutation $\mathbf{i}$, any optimally mixed block decomposition of $\Theta_{\mathbf{i}}$ to this fixed permutation has the same number of the $(s,s)$ blocks, which is called the degree of optimally mixed block decompositions of $\Theta$ to the permutation $\mathbf{i}$ and is denoted by $d_{\mathbf{i}}$.  Let
\begin{eqnarray*}
\mathfrak{A}_{\mathbf{i}}=\{\mathbf{A}_{d,\mathbf{i}}\mid\text{all inner-couplings of $\mathbf{A}_{d,\mathbf{i}}$ are positive and }d=d_{\mathbf{i}}\}.
\end{eqnarray*}
Then, $\mathbf{A}_{d,\mathbf{i}}$ is an optimally mixed block decomposition of $\Theta$ to the permutation $\mathbf{i}$ if and only if $\mathbf{A}_{d,\mathbf{i}}\in\mathfrak{A}_{\mathbf{i}}$.
Let
\begin{eqnarray*}
\mathfrak{d}=\min\{d_{\mathbf{i}}\mid \mathbf{i}\text{ is a permutation of }(1,2,\cdots,k)\}
\end{eqnarray*}
and
\begin{eqnarray*}
\mathfrak{S}=\{\mathbf{j}\mid \mathbf{j}\text{ is a permutation of }(1,2,\cdots,k)\text{ and }d_{\mathbf{j}}=\mathfrak{d}\}.
\end{eqnarray*}
Then, $\mathfrak{S}\not=\emptyset$.  $\mathbf{A}_{d_{\mathbf{j}},\mathbf{j}}$ is called an {\bf optimally mixed block decomposition of $\Theta$} if $\mathbf{j}\in\mathfrak{S}$.
By our definitions, an optimally mixed block decomposition of $\Theta$, say $\mathbf{A}_{d_{\mathbf{j}},\mathbf{j}}$, is the one that, the number of the $(s,s)$ blocks of $\mathbf{A}_{d_{\mathbf{j}},\mathbf{j}}$ is the smallest in all decompositions of $\Theta_{\mathbf{i}}$ for all permutations $\mathbf{i}$, whose inner-couplings are all positive.  Let
\begin{eqnarray*}
\mathfrak{A}=\{\mathbf{A}_{d_{\mathbf{i}},\mathbf{i}}\mid\mathbf{A}_{d_{\mathbf{i}},\mathbf{i}} \text{ is an optimally mixed block decomposition to $\mathbf{i}$ and }d_{\mathbf{i}}=\mathfrak{d}\}.
\end{eqnarray*}
Then, $\mathbf{A}_{d_{\mathbf{i}},\mathbf{i}}$ is an optimally mixed block decomposition of $\Theta$ if and only if $\mathbf{A}_{d_{\mathbf{i}},\mathbf{i}}\in\mathfrak{A}$.  Clearly, the number of $(s,s)$ blocks in every optimally mixed block decomposition is the same, and this number is called the degree of optimally mixed block decompositions of $\Theta$ and is denoted by $d$.  Without loss of generality, in what follows, we always assume that $\mathbf{A}_{d_{\mathbf{o}},\mathbf{o}}\in\mathfrak{A}$, where $\mathbf{o}=(1,2,\cdots,k)$.  For the sake of simplicity, we re-denote $\mathbf{A}_{d_{\mathbf{o}},\mathbf{o}}$ and $d_{\mathbf{o}}$ by $\mathbf{A}_d$ and $d$, respectively.

\medskip

Since all inner-couplings of an optimally mixed block decomposition, say $\mathbf{A}_d$, are {\bf positive}, for the inter-couplings $\{\beta_{i,j}\}$, either
\begin{enumerate}
\item[$(1)$]  there exists an $(s,s)$ block $\Theta_{s,s}$ such that $\beta_{i,j}$ are negative for all $i\in(a_{s-1}, a_{s}]_{\bbn}$ and $j\not\in(a_{s-1}, a_{s}]_{\bbn}$ or
\item[$(2)$]  $\beta_{i,j}$ are still mixed for all $(i,j)\in\mathcal{K}_{s,t,\mathbf{a}_d}$ and all $1\leq s<t\leq d$.
\end{enumerate}
In the case~$(1)$, $\mathbf{A}_d$ is called repulsive-mixed while in the case~$(2)$, $\mathbf{A}_d$ is called total-mixed.  If there exists an optimally mixed block decomposition that is repulsive-mixed then the mixed couplings $\{\beta_{i,j}\}$ are called {\bf repulsive-mixed} while if all optimally mixed block decompositions are total-mixed then the mixed couplings $\{\beta_{i,j}\}$ are called {\bf total-mixed}.

\medskip

From the definitions above,  for  purely attractive couplings,  its optimal  block decomposition  has the degree $d=1$, while for the purely repulsive couplings the degree of its optimal block decomposition is $k$.  Clearly, the optimal block decompositions of the coefficient matrix $\Theta$ for the purely attractive couplings and the purely repulsive couplings, respectively, are unique up to all permutations of $(1,2,\cdots,k)$.  In what follows, for the sake of simplicity, the optimally mixed block decompositions of mixed couplings are also called their optimal block decompositions.   Thus by the definition of optimal block decompositions, the couplings $\{\beta_{i,j}\}$ can be {\bf classified into four classes}: the purely attractive case, the purely repulsive case, the repulsive-mixed case and the total-mixed case.

\medskip

Let us next define eventual block decompositions.  We rewrite $\mathbf{A}_d$ as
\begin{eqnarray*}
\mathbf{A}_d=[\Theta_{t,s}]_{t,s=1,2,\cdots,d}
\end{eqnarray*}
and define the {\bf interaction forces} between $\Theta_{s,s}$ and $\Theta_{t,t}$ as
\begin{eqnarray*}
\mathfrak{F}_{s,t}^0=\sup_{R_{s,t}>>1}\sum_{(i,j)\in\mathcal{K}_{s,t,\mathbf{a}_d};s\not=t}
&\bigg(&\sum_{\lambda_i=\lambda_j}
C_{i,j}^{s,t}\beta_{i,j}(\frac{1}{R_{s,t}})^{N-1-\alpha}e^{-2\sqrt{\lambda_i}R_{s,t}}\notag\\
&&+\sum_{\lambda_i\not=\lambda_j}C_{i,j}^{s,t}\beta_{i,j}(\frac{1}{R_{s,t}})^{N-1}
e^{-2\min\{\sqrt{\lambda_i},\sqrt{\lambda_j}\}R_{s,t}}\bigg),
\end{eqnarray*}
where $\alpha=1$ for $N=1$ and $\alpha=\frac{1}{2}$ for $N=2,3$.  Let
\begin{eqnarray*}
\mathbf{A}_{d^1}^1=[\Theta_{t,s}^1]_{t,s=1,2,\cdots,d^1}
\end{eqnarray*}
be such a decomposition:  $\Theta_{t,s}^1$ are consisted by $\Theta_{i,j}$ such that all interaction forces $\mathfrak{F}_{i,j}^0$ between $\Theta_{i,i}$ and $\Theta_{j,j}$ in $\Theta_{t,s}^1$ are positive.  Without loss of generality, we denote $\Theta_{t,s}^1$ by
\begin{eqnarray*}
\Theta_{t,s}^1=[\Theta_{i,j}]_{(i,j)\in\mathcal{K}_{t,s,\mathbf{a}^1_{d^1}}},
\end{eqnarray*}
Where
\begin{eqnarray*}
\mathcal{K}_{t,s,\mathbf{a}^1_{d^1}}=(a^1_{t-1}, a^1_{t}]_{\bbn}\times(a^1_{s-1}, a^1_{s}]_{\bbn}
\end{eqnarray*}
with $\mathbf{a}^1_{d^1}=(a^1_0,a^1_1,\cdots,a^1_{d^1})$, $(a^1_{t-1}, a^1_{t}]_{\bbn}=(a^1_{t-1}, a^1_{t}]\cap\bbn$ and $0=a^1_0<a^1_1<\cdots<a^1_{d^1-1}<a^1_{d^1}=d$.  We then define the interaction forces between $\Theta_{s,s}^1$ and $\Theta_{t,t}^1$ as
\begin{eqnarray*}
\mathfrak{F}_{s,t}^1=\sum_{(i,j)\in\mathcal{K}_{t,s,\mathbf{a}^1_{d^1}}}\mathfrak{F}_{i,j}^0.
\end{eqnarray*}
We repeat these two steps over and over again until we can not further group in this way any more.  Without loss of generality, we assume that these two steps can be repeated $\tau$ times.  Moreover, for the sake of simplicity, we re-denote the optimal block  decomposition by $\mathbf{A}_{d^0}^0$.  Then we will obtain a sequence of decompositions
\begin{eqnarray*}
\mathbf{A}_{d^\varsigma}^\varsigma=[\Theta_{t,s}^\varsigma]_{t,s=1,2,\cdots,d^\varsigma}
\end{eqnarray*}
with
\begin{eqnarray*}
\Theta_{t,s}^\varsigma=[\Theta_{i,j}^{\varsigma-1}]_{(i,j)\in\mathcal{K}_{t,s,\mathbf{a}^\varsigma_{d^\varsigma}}}
\end{eqnarray*}
and $1\leq\varsigma\leq\tau$,
\begin{eqnarray*}
\mathcal{K}_{t,s,\mathbf{a}^\varsigma_{d^\varsigma}}=(a^\varsigma_{t-1}, a^\varsigma_{t}]_{\bbn}\times(a^\varsigma_{s-1}, a^\varsigma_{s}]_{\bbn}
\end{eqnarray*}
with $\mathbf{a}^\varsigma_{d^\varsigma}=(a^\varsigma_0,a^\varsigma_1,\cdots,a^\varsigma_{d^\varsigma})$, $(a^\varsigma_{t-1}, a^\varsigma_{t}]_{\bbn}=(a^\varsigma_{t-1}, a^\varsigma_{t}]\cap\bbn$ and $0=a^\varsigma_0<a^\varsigma_1<\cdots<a^\varsigma_{d^\varsigma-1}<a^\varsigma_{d^\varsigma}=d^{\varsigma-1}$, and a sequence $1\leq d^{\tau}< d^{\tau-1}< \cdots< d^1< d^0=d$.  $\mathbf{A}_{d^\tau}^\tau$ is called an {\bf eventual block  decomposition} of $\mathbf{A}_{d^0}^0$, and the number of $(s,s)$ blocks $\Theta_{s,s}^\tau$ is called the degree of $\mathbf{A}_{d^\tau}^\tau$ and is denoted by $m$.  To obtain all eventual block  decompositions of $\mathbf{A}_{d^0}^0$, for the $\varsigma_{th}$ decomposition $\mathbf{A}_{d^\varsigma}^\varsigma$, $0\leq\varsigma\leq\tau-1$, we should write down all next decompositions $\mathbf{A}_{d^{\varsigma+1}}^{\varsigma+1}$ in the above way under the action of permutations.  Clearly, for other optimal block decompositions, we can obtain their eventual block decompositions in the same way.  By our definitions, the degrees of eventual block decompositions of the purely repulsive case and the repulsive-mixed cases are always strictly large than $1$, while the degrees of eventual block decompositions of the purely attractive case always equal to $1$.

\medskip

In the $(s,s)$ block $\Theta_{s,s}=[\beta_{i,j}]_{(i,j)\in\mathcal{K}_{s,s,\mathbf{a}_d}}$ of $\mathbf{A}_d$, either $\{i\not=j,(i,j)\in\mathcal{K}_{s,s,\mathbf{a}_d}\}\not=\emptyset$ or $\{i\not=j,(i,j)\in\mathcal{K}_{s,s,\mathbf{a}_d}\}=\emptyset$.  Without loss of generality, we assume that $\{i\not=j,(i,j)\in\mathcal{K}_{s,s,\mathbf{a}_d}\}\not=\emptyset$ for $s=1,2,\cdots,s_0$ and $\{i\not=j,(i,j)\in\mathcal{K}_{s,s,\mathbf{a}_d}\}=\emptyset$ for $s=s_0+1,\cdots,d$ with an $s_0\in\{0,1,2,\cdots,d\}$.  For every $d\leq\gamma\leq k$, there exists a unique $0\leq s^*\leq s_0$ such that $a_{s^*}\leq k-\gamma<a_{s^*+1}$.  Now, our results for the general $k$-coupled system~\eqref{eqn0001} can be stated as follows.
\begin{theorem}\label{thm0002}
Let $N=1,2,3$ and $k\geq 3$.  Suppose that the degree of optimal block decompositions of the coefficient matrix $\Theta$ is $d$.  Then,
\begin{enumerate}
\item[$(1)$]  if all eventual block decompositions satisfy $m=1$ then for every $d\leq\gamma\leq k$, there exist $\widehat{\beta}_0>\beta_0>0$ such that if
    \begin{enumerate}
    \item[$(i)$]  $\beta_{i,j}>\widehat{\beta}_0$ and $|\beta_{i,j}-\beta_{i,l}|<<1$ for all $(i,j), (i,l)\in\mathcal{K}_{s,s,\mathbf{a}_d}$ with $i\not=j$, $i\not=l$ and $j\not=l$, and $i,j,l\leq k-\gamma+1$,
    \item[$(ii)$]  $\beta_{i,j}<\beta_0$ for all other $(i,j)$ with $i\not=j$ that are not contained in $(i)$,
    \end{enumerate}
    then \eqref{eqn0001} has a ground state with the Morse index $\gamma$, provided that $|\lambda_i-\lambda_j|<<1$ for all $i,j\in\mathcal{K}_{s,s,\mathbf{a}_d}$ and $i\not=j$ with $0\leq s\leq s^*$ satisfying $a_s-a_{s-1}\geq3$ and for all $i,j\in\mathcal{K}_{s^*+1,s^*+1,\mathbf{a}_d}$,$i\not=j$ and $i,j\leq k-\gamma+1$ satisfying $k-\gamma-a_{s^*}\geq3$.  In particular, in the purely attractive case, for every $1\leq\gamma\leq k$, \eqref{eqn0001} has a ground state with the Morse index $\gamma$.
\item[$(2)$]  Suppose $\beta_{i,j}=\delta^{t_{i,j}}\widehat{\beta}_{i,j}$, where $\delta>0$ is a parameter and $t_{i,j}$, $\widehat{\beta}_{i,j}$ are absolute constants.  If the couplings $\beta_{i,j}$ are total-mixed, $t_{i,j}=t_0$ for all $(i,j)\in\mathcal{K}_{s,s,\mathbf{a}_d}$ and all $0\leq s\leq s_0$, $t_{0}< t_{min,int,+}$, $t_{max,-}<t_{min,+}$ and
    \begin{eqnarray*}
\min\{\sqrt{\lambda_{i_0}}, \sqrt{\lambda_{j_0}}\}\geq\min\{\sqrt{\lambda_{i_0'}}, \sqrt{\lambda_{j_0'}}\}
\end{eqnarray*}
for all $(i_0,j_0)$ and $(i_0',j_0')$ with $\beta_{i_0,j_0}>0>\beta_{i_0',j_0'}$,
then $\mathcal{C}_{\mathcal{N}}$ can not be attained for $\delta>0$ sufficiently small.  That is, \eqref{eqn0001} has no ground states.  Here, $t_{max,-}=\max\{t_{i,j}\mid\widehat{\beta}_{i,j}<0\}$, $t_{min,+}=\min\{t_{i,j}\mid\widehat{\beta}_{i,j}>0\}$,
and
$$
t_{min,int,+}=\min\{t_{i,j}\mid\widehat{\beta}_{i,j}>0\text{ and }\beta_{i,j}\text{ is a inter-coupling}\}.
$$
\item[$(3)$]  If the couplings $\beta_{i,j}$ are repulsive-mixed or purely repulsive, then $\mathcal{C}_{\mathcal{N}}$ can not be attained, provided that the coefficient matrix $\Theta= (\beta_{i,j})$ is positively definite.  That is, \eqref{eqn0001} has no ground states.
\end{enumerate}
\end{theorem}

\begin{remark}
\begin{enumerate}
\item[$(a)$]  The existence result yields a very interesting consequence:  The degree of optimal block decompositions determines the lower bound of the Morse index of the ground states of \eqref{eqn0001}.  According to our definitions, the degree of optimal block decompositions is the smallest number of the groups, which are made up by the components $\{u_j\}$ such that they are all attractive to each others in these groups.  This implies that, in Bose-Einstein condensates for multi-species condensates, the components $\{u_j\}$ will huddle as much as possible.  On the other hand, as one can see by comparing Theorems~\ref{coro0001} and \ref{thm0003}, the existence conditions of the four-coupled system~\eqref{eqnewnew0001} in the total-mixed case $(H)$ at (\ref{Hdef}) are much stronger than that of the three-coupled system~\eqref{eqnew0001}.  This is caused by the fact that the four-coupled system~\eqref{eqnewnew0001} has more $(s,s)$ blocks in its optimal block decompositions in the total-mixed case $(H)$ at (\ref{Hdef}), which needs more interaction forces to be positive to grantee the existence of ground states.  Thus, it seems that the ground states are harder to exist if its optimal block decompositions has more $(s,s)$ blocks.  In the extremal case in this direction, i.e., the purely repulsive case or the repulsive-mixed cases, there are no ground states.
\item[$(b)$]  As we pointed out in $(c)$ of Remark~\ref{rmk0003}, some existence and nonexistence results for \eqref{eqn0001} in some very special cases have been obtained in the literature, see, for example, \cite{LW05,LW10,ST16}.
\item[$(c)$]  Another interesting fact is that the Morse index of ground states is related to the number of eigenvalues of the coefficient matrix.  To understand this relation, we use the four-coupled system~\eqref{eqnewnew0001} in the total-mixed case $(H)$ at (\ref{Hdef}) as an example.  Indeed, under the conditions of $(1)$ of Theorem~\ref{thm0003}, the coefficient matrix is nonsingular.  Moreover, in $(i)$ of $(1)$ of Theorem~\ref{thm0003} the coefficient matrix has four positive eigenvalues, while in $(ii)$ of $(1)$ of Theorem~\ref{thm0003} the coefficient matrix has three positive eigenvalues and one negative eigenvalue.  Since roughly speaking, the superlinear nonlinearities are determined by the coefficient matrix and they ``generate'' the negative part in the second derivative of the functional, $\gamma$ positive eigenvalues of the coefficient matrix will ``generate'' $\gamma$ Morse index of the ground states.
\end{enumerate}
\end{remark}

Since the main ideas in proving these three Theorems are similar, to make our proof easier to follow and to avoid unnecessary complicated calculations, we only give a complete proof of Theorem~\ref{coro0001} in section~4.  We will also sketch the proof of Theorems~\ref{thm0003} and \ref{thm0002} by pointing out necessary changes in section~5.
\medskip

\noindent{\bf\large Notations.} Throughout this paper, $C$ and $C'$ are indiscriminately used to denote various absolutely positive constants.  $a\sim b$ means that $C'b\leq a\leq Cb$ and $a\lesssim b$ means that $a\leq Cb$.

\section{Three-coupled system~\eqref{eqnew0001}}

\subsection{Some preliminaries}
In this section, we state some well-known results which will be frequently used in proving Theorem~\ref{coro0001}.
Let $w_j$ be the unique solution of the following scalar field equation
\begin{eqnarray}\label{eqnew0013}
\left\{\aligned&-\Delta u +\lambda_j u =\mu_ju^3\quad\text{in }\bbr^N,\\
&u>0\quad\text{in }\bbr^N,\quad u(0)=\max_{x\in\bbr^N}u(x),\\
&u(x)\to0\quad\text{as }|x|\to+\infty.
\endaligned\right.
\end{eqnarray}
Then, $w_j$, satisfying
\begin{eqnarray}\label{eqnew9998}
w_j(|x|)\sim |x|^{-\frac{N-1}{2}}e^{-\sqrt{\lambda_j}|x|}\quad\text{as }|x|\to+\infty,
\end{eqnarray}
is radially symmetric and strictly decreasing in $|x|$.
The energy functional of \eqref{eqnew0013} in $\mathcal{H}_j$ is given by
\begin{eqnarray}\label{eqnew0099}
\mathcal{E}_j(u)=\frac{1}{2}\|u\|_{\lambda_j}^2-\frac{\mu_j}{4}\|u\|_4^4
\end{eqnarray}
and the corresponding  Nehari manifold is
\begin{eqnarray*}
\mathcal{N}_j=\{u\in\mathcal{H}_j\backslash\{0\}\mid\mathcal{E}_j'(u)u=0\}.
\end{eqnarray*}
We need the following estimate which will be used frequently in this paper. The proof is technical and thus delayed to appendix.
\begin{lemma}\label{lemn0010}
Let $N=1,2,3$ and $w_j$ be the unique solution of \eqref{eqnew0013}.  Suppose $e_1\in\bbr^N$ such that $|e_1|=1$.  Then as $R\to+\infty$,
\begin{eqnarray*}
\int_{\bbr^N}w_i^2(x)w_j^2(x-Re_1)dx\sim \left\{\aligned R^{1-N}e^{-2\min\{\sqrt{\lambda_i}, \sqrt{\lambda_j}\}R},\quad \lambda_i\not=\lambda_j;\\
R^{1+\alpha-N}e^{-2\sqrt{\lambda}R},\quad \lambda_i=\lambda_j=\lambda,\endaligned
\right.
\end{eqnarray*}
where $\alpha=1$ for $N=1$ and $\alpha=\frac{1}{2}$ for $N=2,3$.
\end{lemma}

We also define energy functionals, which are of class $C^2$ in $\mathcal{H}_{i,j}=\mathcal{H}_i\times\mathcal{H}_j$, as follows:
\begin{eqnarray}\label{eqn0104}
\mathcal{E}_{i,j}(\overrightarrow{\phi})=\frac{1}{2}(\|\phi_i\|_{\lambda_i}^2+\|\phi_j\|_{\lambda_j}^2)
-\frac{1}{4}(\mu_i\|\phi_i\|_{4}^4+\mu_j\|\phi_j\|_{4}^4)-\frac{\beta_{i,j}}{2}\|\phi_i\phi_j\|_2^2,
\end{eqnarray}
where $\overrightarrow{\phi}=(\phi_i, \phi_j)$ and $(i,j)$ equals to $(1,2)$, $(1,3)$ or $(2,3)$.  Positive critical points of $\mathcal{E}_{i,j}(\overrightarrow{\phi})$ are equivalent to the solutions of the following system
\begin{equation}\label{eqn0099}
\left\{\aligned&-\Delta u_i+\lambda_iu_i=\mu_iu_i^3+\beta_{i,j} u_j^2u_i\quad\text{in }\bbr^N,\\
&-\Delta u_j+\lambda_ju_j=\mu_ju_j^3+\beta_{i,j} u_i^2u_j\quad\text{in }\bbr^N,\\
&u_i,u_j>0\quad\text{in }\bbr^N,\quad u_i(x),u_j(x)\to0\quad\text{as }|x|\to+\infty.\endaligned\right.
\end{equation}
We define the Nehari manifold of $\mathcal{E}_{i,j}(\overrightarrow{\phi})$ as follows:
\begin{eqnarray*}
\mathcal{N}_{i,j}=\{\overrightarrow{\phi}\in\widetilde{\mathcal{H}}_{i,j}\mid \overrightarrow{\widehat{\mathcal{G}}}_{i,j}(\overrightarrow{\phi})=(\widehat{\mathcal{G}}_i(\overrightarrow{\phi}), \widehat{\mathcal{G}}_j(\overrightarrow{\phi}))=\overrightarrow{0}\},
\end{eqnarray*}
where $\widetilde{\mathcal{H}}_{i,j}=(\mathcal{H}_i\backslash\{0\})\times(\mathcal{H}_j\backslash\{0\})$, $\widehat{\mathcal{G}}_j(\overrightarrow{\phi})=\|\phi_j\|_{\lambda_j}^2
-\mu_j\|\phi_j\|_{4}^4-\beta_{i,j}\|\phi_i\phi_j\|_2^2$ and $\widehat{\mathcal{G}}_i(\overrightarrow{\phi})=\|\phi_i\|_{\lambda_i}^2
-\mu_i\|\phi_i\|_{4}^4-\beta_{i,j}\|\phi_i\phi_j\|_2^2$.  Let
\begin{eqnarray}\label{eqn0110}
\mathcal{C}_{\mathcal{N}_{i,j}}=\inf_{\mathcal{N}_{i,j}}\mathcal{E}_{i,j}(\overrightarrow{\phi}).
\end{eqnarray}
Then, $\mathcal{C}_{\mathcal{N}_{i,j}}$ is well defined and nonnegative for all $i\not = j$.  Moreover, there exists $0<\beta_*<\sqrt{\mu_i \mu_j}$ such that if $0<\beta_{i,j}<\beta_*<\sqrt{\mu_i \mu_j}$ then $\mathcal{C}_{\mathcal{N}_{i,j}}$ is attained by $\overrightarrow{\varphi}^{i,j}$ which is positive and radially symmetric (cf. \cite[Theorem~1.2]{CZ13}).  Clearly, $\overrightarrow{\varphi}^{i,j}$ is also a solution of \eqref{eqn0099}.  Applying the comparison principle as for \cite[(4.6) and (4.7)]{LW05} yields that
\begin{eqnarray}\label{eqnew9999}
\varphi^{i,j}_{i}(|x|)\sim |x|^{-\frac{N-1}{2}}e^{-\sqrt{\lambda_i}|x|}\quad\text{as }|x|\to+\infty.
\end{eqnarray}

\subsection{Ground states with the Morse index 3}
In this section, we will study the existence of the ground states of \eqref{eqnew0001} with the Morse index 3, in the total-mixed case~$(d)$: $\beta_{1,2}>0$, $\beta_{1,3}>0$ and $\beta_{2,3}<0$.

Recall the definition of the Nehari manifold $\mathcal{N}$ at (\ref{Neh1}) and the least energy value
 $\mathcal{C}_{\mathcal{N}}=\inf_{\mathcal{N}}\mathcal{E}(\overrightarrow{u})$ at (\ref{CN}).
 Using $(w_{1,-R},w_2,w_{3,R})$ as a test function and calculating similarly in the proof of \cite[Theorem~1]{LW05} yields
\begin{eqnarray}\label{eqnew0098}
\mathcal{C}_{\mathcal{N}}\leq\sum_{j=1}^3\mathcal{E}_j(w_j),
\end{eqnarray}
where $w_j$ and $\mathcal{E}_j(u)$ are given by \eqref{eqnew9998} and \eqref{eqnew0099}, respectively, and $w_{j,z}=w_j(x+z)$.
\begin{lemma}\label{lemn0001}
There exists $\beta_0>0$ such that $\mathcal{N}$ contains a $(PS)$ sequence at the least energy value $\mathcal{C}_{\mathcal{N}}$ for $0<\beta_{1,2},\beta_{1,3}<\beta_0$ and $\beta_{2,3}<0$.
Moreover, any positive minimizer of $\mathcal{E}(\overrightarrow{u})$ on $\mathcal{N}$ is a ground state of \eqref{eqnew0001} with the Morse index 3.
\end{lemma}
\begin{proof}
The proof is standard, so we only sketch it.  By a standard argument, there exists $\beta_0>0$ such that $1\lesssim\|u_j\|_4^4$ for all $\overrightarrow{u}\in\mathcal{N}$ with $\sum_{j=1}^3\|u_j\|_{\lambda_j}^2\leq8\sum_{j=1}^3\mathcal{E}_j(w_j)$ and $j=1,2,3$ for $0<\beta_{1,2},\beta_{1,3}<\beta_0$ and $\beta_{2,3}<0$.  Thus, the matrix $\Xi=[\beta_{i,j}\|u_iu_j\|_2^2]_{i,j=1,2,3}$ is strictly diagonally dominant for $\overrightarrow{u}\in\mathcal{N}$, with $\sum_{j=1}^3\|u_j\|_{\lambda_j}^2\leq8\sum_{j=1}^3\mathcal{E}_j(w_j)$, where $\beta_{j,j}=\mu_j$.  It follows that $\Xi$ is positively definite, with $1\lesssim|\text{det}(\Xi)|$.  Thus, applying the implicit function theorem, the Ekeland variational principle and the Taylor expansion in a standard way yields that, $\mathcal{N}$ contains a $(PS)$ sequence at the least energy value $\mathcal{C}_{\mathcal{N}}$.  Since $1\lesssim|\text{det}(\Xi)|$ for $\overrightarrow{u}\in\mathcal{N}$ with $\sum_{j=1}^3\|u_j\|_{\lambda_j}^2\leq8\sum_{j=1}^3\mathcal{E}_j(w_j)$, for any positive minimizer of $\mathcal{E}(\overrightarrow{u})$ on $\mathcal{N}$, say $\overrightarrow{v}$, $\mathcal{H}=\mathcal{T}_{\overrightarrow{v}}\mathcal{N}
\bigoplus(\bbr\overrightarrow{v}_1\times\bbr\overrightarrow{v}_2\times\bbr\overrightarrow{v}_3)$, where $\mathcal{T}_{\overrightarrow{v}}\mathcal{N}$ is the tangent space of $\mathcal{N}$ as $\overrightarrow{v}$, $\overrightarrow{v}_1=(v_1,0,0)$, $\overrightarrow{v}_2=(0,v_2,0)$ and $\overrightarrow{v}_3=(0,0,v_3)$.  Since $\overrightarrow{v}$ is a positive minimizer of $\mathcal{E}(\overrightarrow{u})$ on $\mathcal{N}$, $\mathcal{E}''(\overrightarrow{v})(\overrightarrow{h},\overrightarrow{h})\geq0$ for all $\overrightarrow{h}\in\mathcal{T}_{\overrightarrow{v}}\mathcal{N}$.  It follows that the Morse index of $\overrightarrow{v}$ is less than or equal to $3$.  On the other hand, since
\begin{eqnarray*}
\mathcal{E}''(\overrightarrow{v})(\overrightarrow{v}_i,\overrightarrow{v}_i)=\|v_i\|_{\lambda_i}^2-3\mu_i\|v_i\|_4^4
-\sum_{j=1,j\not=i}^3\beta_{i,j}\|v_iv_j\|_2^2
=-2\mu_i\|v_i\|_4^4<0
\end{eqnarray*}
for all $i=1,2,3$, the Morse index of $\overrightarrow{v}$ is greater than or equal to $3$.  Thus, $\overrightarrow{v}$ is a ground state of \eqref{eqnew0001} with the Morse index 3.
\end{proof}

By Lemma~\ref{lemn0001}, to prove the existence of the ground states of \eqref{eqnew0001} with the Morse index 3 in the total-mixed case, it is sufficient to prove the existence of a positive minimizer of $\mathcal{E}(\overrightarrow{u})$ on the Nehari manifold $\mathcal{N}$.  We start by the following energy estimate.
\begin{lemma}\label{lem0001}
Let $\beta_{1,2}>0$, $\beta_{1,3}>0$ and $\beta_{2,3}<0$.  If  $\lambda_1<\min\{\lambda_2,\lambda_3\}$ then
\begin{eqnarray*}
\mathcal{C}_{\mathcal{N}}<\min\{\mathcal{C}_{\mathcal{N}_{1,2}}+\mathcal{E}_3(w_3), \mathcal{C}_{\mathcal{N}_{1,3}}+\mathcal{E}_2(w_2)\}
\end{eqnarray*}
for $\beta_{1,2},\beta_{1,3}<\beta_0$, where $\beta_0$ is given by Lemma~\ref{lemn0001}, $\mathcal{C}_{\mathcal{N}_{i,j}}$ are given by \eqref{eqn0110} and $\mathcal{C}_{\mathcal{N}}=\inf_{\mathcal{N}}\mathcal{E}(\overrightarrow{u})$.
\end{lemma}
\begin{proof}
We only give the proof of $\mathcal{C}_{\mathcal{N}}<\mathcal{C}_{\mathcal{N}_{1,2}}+\mathcal{E}_3(w_3)$ since the proof of the other inequality is similar.  For the sake of simplicity, we denote $\varphi^{1,2}_j$ by $\varphi_j$, where $\overrightarrow{\varphi}^{1,2}=(\varphi^{1,2}_1,\varphi^{1,2}_2)$ is a ground state of \eqref{eqn0099} for $(i,j)=(1,2)$.  Let $w_{3,R}=w_3(x-Re_1)$ where $e_1\in\bbr^N$ satisfying $|e_1|=1$.  We consider the following system
\begin{eqnarray}\label{eqnew0011}
\left\{\aligned&\|\varphi_1\|_{\lambda_1}^2=\mu_1\|\varphi_1\|_4^4t_1^2(R)+\beta_{1,2}\|\varphi_1\varphi_2\|_2^2t_2^2(R)
+\beta_{1,3}\|\varphi_1w_{3,R}\|_2^2t_3^2(R),\\
&\|\varphi_2\|_{\lambda_2}^2=\mu_2\|\varphi_2\|_4^4t_2^2(R)+\beta_{1,2}\|\varphi_1\varphi_2\|_2^2t_1^2(R)
+\beta_{2,3}\|\varphi_2w_{3,R}\|_2^2t_3^2(R),\\
&\|w_3\|_{\lambda_3}^2=\mu_3\|w_3\|_4^4t_3^2(R)+\beta_{1,3}\|\varphi_1w_{3,R}\|_2^2t_1^2(R)
+\beta_{2,3}\|\varphi_2w_{3,R}\|_2^2t_2^2(R).
\endaligned\right.
\end{eqnarray}
Clearly, $\{t_j(R)\}$, $j=1,2,3$, are bounded for sufficiently large $R>0$ and $t_j(R)\to1$ as $R\to+\infty$.  Moreover, since $\|\varphi_jw_{3,R}\|_2^2\to0$ as $R\to+\infty$ for $j=1,2$, by taking $\beta_0$ in Lemma~\ref{lemn0001} sufficiently small if necessary, the above linear system is uniquely solvable for $\beta_{1,2}<\beta_0$.  Its unique solution $(t_1^2(R),t_2^2(R),t_3^2(R))$ is given by
\begin{eqnarray*}
t_j^2(R)=1-\frac{(1+o_R(1))(\beta_{j,3}\|\varphi_jw_{3,R}\|_2^2\mu_i\|\varphi_i\|_4^4
-\beta_{i,3}\|\varphi_iw_{3,R}\|_2^2\beta_{1,2}\|\varphi_1\varphi_2\|_2^2)}
{\prod_{l=1}^2\mu_l\|\varphi_l\|_4^4-\beta_{1,2}^2\|\varphi_1\varphi_2\|_2^4}
\end{eqnarray*}
for $(i,j)$ equals to $(1,2)$ or $(2,1)$ and
\begin{eqnarray*}
t_3^2(R)=1-\frac{1+o_R(1)}{\mu_3\|w_{3}\|_4^4}(\beta_{1,3}\|\varphi_1w_{3,R}\|_2^2+\beta_{2,3}\|\varphi_2w_{3,R}\|_2^2).
\end{eqnarray*}
Here, $o_R(1)\to0$ as $R\to+\infty$.  Since $\beta_{1,2}>0$, \eqref{eqnew9999} holds for $\varphi_j$, $j=1,2$.  Thus, by Lemma~\ref{lemn0010} and $\lambda_1<\min\{\lambda_2,\lambda_3\}$,
\begin{eqnarray}\label{eqnew8989}
\|\varphi_1w_{3,R}\|_2^2\sim R^{1-N}e^{-2\sqrt{\lambda_1}R}\quad\text{as }R\to+\infty.
\end{eqnarray}
By Lemma~\ref{lemn0010} once more, as $R\to+\infty$,
\begin{eqnarray}\label{eqnew8832}
\|\varphi_2w_{3,R}\|_2^2\sim \left\{\aligned R^{1-N}e^{-2\min\{\sqrt{\lambda_2}, \sqrt{\lambda_3}\}R},\quad \lambda_2\not=\lambda_3;\\
R^{1+\alpha-N}e^{-2\sqrt{\lambda}R},\quad \lambda_2=\lambda_3=\lambda,\endaligned
\right.
\end{eqnarray}
where $\alpha=1$ for $N=1$ and $\alpha=\frac{1}{2}$ for $N=2,3$.
Since $(t_1(R),t_2(R),t_3(R))$ satisfies \eqref{eqnew0011}, we can test $\mathcal{C}_{\mathcal{N}}$ by
$$
(t_1(R)\varphi_1,t_2(R)\varphi_2,t_3(R)w_{3,R})
$$
and estimate it by \eqref{eqnew8989} as follows:
\begin{eqnarray}
\mathcal{C}_{\mathcal{N}}&\leq&\frac{1}{4}(\sum_{j=1}^2t_j^2(R)\|\varphi_j\|_{\lambda_j}^2+t_3^2(R)\|w_{3,R}\|_{\lambda_3}^2)\notag\\
&\leq&\mathcal{C}_{\mathcal{N}_{1,2}}+\mathcal{E}_3(w_3)-C\beta_{1,3}R^{1-N}e^{-2\sqrt{\lambda_1}R}
-C'\beta_{2,3}\|\varphi_2w_{3,R}\|_2^2\label{eqnew0012}
\end{eqnarray}
By \eqref{eqnew8832} and taking $R>0$ sufficiently large in \eqref{eqnew0012}, it follows from $\lambda_1<\min\{\lambda_2,\lambda_3\}$ that
\begin{eqnarray*}
\mathcal{C}_{\mathcal{N}}<\mathcal{C}_{\mathcal{N}_{1,2}}+\mathcal{E}_3(w_3),\label{eq0003}
\end{eqnarray*}
which completes the proof.
\end{proof}

\begin{remark}\label{rmk0001}
As that in the proof of Lemma~\ref{lem0001}, if we use $(w_i,w_j)$ as a test function of $\mathcal{C}_{\mathcal{N}_{i,j}}$ where $(i,j)$ equals to $(1,2)$ or $(1,3)$, then by taking $\beta_0>0$ sufficiently small if necessary,
\begin{eqnarray*}
\mathcal{C}_{\mathcal{N}_{i,j}}\leq\mathcal{E}_i(w_i)+\mathcal{E}_j(w_j)-\frac{\beta_{i,j}}{2}\|w_iw_j\|_2^2+O(\beta_{i,j}^2)
\end{eqnarray*}
for $0<\beta_{i,j}<\beta_0$.
\end{remark}

Now, we are prepared to prove the following existence result.
\begin{proposition}\label{prop0001}
Let $\beta_{1,2}>0$, $\beta_{1,3}>0$ and $\beta_{2,3}<0$.  If  $\lambda_1<\min\{\lambda_2,\lambda_3\}$ then there exists a positive minimizer of $\mathcal{E}(\overrightarrow{u})$ on $\mathcal{N}$ for $\beta_{1,2},\beta_{1,3}<\beta_0$, where $\beta_0$ is given by Lemma~\ref{lemn0001}.  That is, \eqref{eqnew0001} has a ground state with the Morse index 3.
\end{proposition}
\begin{proof}
By Lemma~\ref{lemn0001}, there exists a $(PS)$ sequence $\{\overrightarrow{u}_n\}$ at the least energy value $\mathcal{C}_{\mathcal{N}}$.  Clearly, $\{\overrightarrow{u}_n\}$ is bounded in $\mathcal{H}$.  Since $1\lesssim\|u_{j,n}\|_4$ for all $j=1,2,3$, by the Lions lemma and the Sobolev embedding theorem, there exist $\{y_{j,n}\}\subset\bbr^N$ such that $u_{j,n}(x+y_{j,n})\rightharpoonup v_{j,\infty}\not=0$ weakly in $H^1(\bbr^N)$ as $n\to\infty$.  We denote $v_{i,j,n}=u_{i,n}(x+y_{j,n})$.  Then, $v_{i,j,n}\rightharpoonup v_{i,j,\infty}$ weakly in $H^1(\bbr^N)$ as $n\to\infty$.  Moreover, $v_{j,j,\infty}=v_{j,\infty}\not=0$ for all $j=1,2,3$.  Since $\{\overrightarrow{u}_n\}$ is a $(PS)$ sequence, it is standard to show that $\overrightarrow{v}_{j,\infty}=(v_{1,j,\infty},v_{2,j,\infty},v_{3,j,\infty})$ is a critical point of $\mathcal{E}(\overrightarrow{u})$ for all $j=1,2,3$.  If for every $j=1,2,3$, we always have $v_{i,j,\infty}=0$ with $i\not=j$, then,
$$
\mathcal{C}_{\mathcal{N}}=\sum_{j=1}^3\frac14\|u_{j,n}\|_{\lambda_j}^2+o_n(1)=\sum_{j=1}^3\frac14\|v_{j,j,n}\|_{\lambda_j}^2+o_n(1)
\geq\sum_{j=1}^3\mathcal{E}_j(w_j)+o_n(1),
$$
which contradicts Lemma~\ref{lem0001} and Remark~\ref{rmk0001} by taking $\beta_0>0$ sufficiently small if necessary.  Thus, without loss of generality, we assume that for $j=1$, one of the following cases must happen:
\begin{enumerate}
\item[$(1)$]  $v_{1,1,\infty}\not=0$, $v_{2,1,\infty}\not=0$ and $v_{3,1,\infty}=0$.
\item[$(2)$]  $v_{1,1,\infty}\not=0$, $v_{2,1,\infty}=0$ and $v_{3,1,\infty}\not=0$.
\item[$(3)$]  $v_{1,1,\infty}\not=0$, $v_{2,1,\infty}\not=0$ and $v_{3,1,\infty}\not=0$.
\end{enumerate}
We first consider the case~$(1)$.  Clearly, $(v_{1,1,\infty}, v_{2,1,\infty})$ is a nontrivial critical point of $\mathcal{E}_{1,2}(\overrightarrow{\phi})$, where $\mathcal{E}_{1,2}(\overrightarrow{\phi})$ is given by \eqref{eqn0104}.  Note that for $j=3$, one of the following cases must happen:
\begin{enumerate}
\item[$(i)$] $v_{1,3,\infty}\not=0$, $v_{2,3,\infty}=0$ and $v_{3,3,\infty}\not=0$.
\item[$(ii)$]  $v_{1,3,\infty}=0$, $v_{2,3,\infty}\not=0$ and $v_{3,3,\infty}\not=0$.
\item[$(iii)$]  $v_{1,3,\infty}=0$, $v_{2,3,\infty}=0$ and $v_{3,3,\infty}\not=0$.
\item[$(iv)$]  $v_{1,3,\infty}\not=0$, $v_{2,3,\infty}\not=0$ and $v_{3,3,\infty}\not=0$.
\end{enumerate}
If the case~$(iv)$ happens, then by a standard argument, $\mathcal{C}_{\mathcal{N}}$ is attained by $\overrightarrow{\hat{v}}_{3,\infty}=(|v_{1,3,\infty}|,|v_{2,3,\infty}|,|v_{3,3,\infty}|)$, which, together with the Harnack inequality and the fact that $\mathcal{N}$ is a natural constraint, implies that there exists a positive minimizer of $\mathcal{E}(\overrightarrow{u})$ on $\mathcal{N}$.  Thus, by Lemma~\ref{lemn0001}, \eqref{eqnew0001} has a ground state with the Morse index 3.  Therefore, without loss of generality, we assume that one of the cases~$(i)$--$(iii)$ must happen in what follows.
Since $v_{3,3,\infty}\not=0$ and $v_{3,3,n}(x)=v_{3,1,n}(x+y_{3,n}-y_{1,n})$,
by the Sobolev embedding theorem, $|y_{3,n}-y_{1,n}|\to+\infty$ as $n\to\infty$.  It follows that for every $R>0$,
\begin{eqnarray*}
\int_{\bbr^N}|v_{1,1,n}|^4dx&\geq&\int_{B_R(0)}|v_{1,1,n}|^4dx+\int_{B_R(y_{3,n}-y_{1,n})}|v_{1,1,n}|^4dx\\
&=&\int_{B_R(0)}|v_{1,1,n}|^4dx+\int_{B_R(0)}|v_{1,3,n}|^4dx.
\end{eqnarray*}
By letting $n\to\infty$ first and $R\to+\infty$ next,
$$
\|v_{1,1,n}\|_4^4\geq\|v_{1,1,\infty}\|_4^4+\|v_{1,3,\infty}\|_4^4+o_n(1).
$$
If the case~$(i)$ happens, then $(v_{1,3,\infty},v_{3,3,\infty})$ is a nontrivial critical point of $\mathcal{E}_{1,3}(\overrightarrow{\phi})$, where $\mathcal{E}_{1,3}(\overrightarrow{\phi})$ is given by \eqref{eqn0104}.  Since it is standard to show that $\|v_{j,3,\infty}\|_4^4\geq\|w_j\|_4^4+o_{\beta_0}(1)$ for sufficiently small $\beta_0$,
\begin{eqnarray*}
\mathcal{C}_{\mathcal{N}}&=&\frac{1}{4}\sum_{j=1}^3\mu_j\|u_{j,n}\|_4^4+\frac{1}{2}\sum_{i=1,i<j}^3\beta_{i,j}\|u_{i,n}u_{j,n}\|_2^2+o_n(1)\\
&\geq&\frac{1}{4}\sum_{j=1}^2\mu_j\|v_{j,1,n}\|_4^4+\frac{\mu_3}{4}\|v_{3,3,n}\|_4^4+o_{\beta_0}(1)+o_n(1)\\
&\geq&\mathcal{C}_{\mathcal{N}_{1,2}}+\mathcal{E}_3(w_3)+\frac{\mu_1}{4}\|w_1\|_4^4+o_{\beta_0}(1)+o_n(1),
\end{eqnarray*}
which contradicts Lemma~\ref{lem0001} for $\beta_{i,j}<\beta_0$ by taking $\beta_0>0$ sufficiently small if necessary.  Here, $o_{\beta_0}(1)\to0$ as $\beta_0\to0$.  The case~$(iii)$ is also impossible since in this case,
\begin{eqnarray*}
\mathcal{C}_{\mathcal{N}}&=&\frac{1}{4}\sum_{j=1}^3\|u_{j,n}\|_{\lambda_j}^2+o_n(1)\\
&=&\frac{1}{4}\sum_{j=1}^2\|v_{j,1,n}\|_{\lambda_j}^2+\frac{1}{4}\|v_{3,3,n}\|_{\lambda_j}^2+o_n(1)\\
&\geq&\mathcal{C}_{\mathcal{N}_{1,2}}+\mathcal{E}_3(w_3)+o_n(1),
\end{eqnarray*}
which still contradicts Lemma~\ref{lem0001}.  Thus, we must have the case~$(ii)$.  If $|y_{1,n}-y_{2,n}|\lesssim1$, then by $|y_{1,n}-y_{3,n}|\to+\infty$ as $n\to\infty$, $|y_{2,n}-y_{3,n}|\to+\infty$ as $n\to\infty$.  It follows from
\begin{eqnarray*}
v_{2,2,n}(x)=v_{2,3,n}(x+y_{2,n}-y_{3,n})
\end{eqnarray*}
that $\|v_{2,3,n}\|_4^4\geq\|v_{2,3,\infty}\|_4^4+\|v_{2,2,\infty}\|_4^4+o_n(1)$.  Then by a similar calculation used in the above arguments,
\begin{eqnarray*}
\mathcal{C}_{\mathcal{N}}\geq\mathcal{C}_{\mathcal{N}_{2,3}}+\mathcal{E}_1(w_1)+o_{\beta_0}(1)+\frac{\mu_2}{4}\|w_2\|_4^4+o_n(1).
\end{eqnarray*}
Since $\beta_{2,3}<0$, it is well known that $\mathcal{C}_{\mathcal{N}_{2,3}}=\sum_{j=2}^3\mathcal{E}_j(w_j)$.  Thus, it is impossible for sufficiently small $\beta_0>0$, owing to Lemma~\ref{lem0001}.  It remains to exclude the case $|y_{1,n}-y_{2,n}|\to+\infty$ as $n\to\infty$.  In this case, it follows from
\begin{eqnarray*}
v_{2,2,n}(x)=v_{2,1,n}(x+y_{2,n}-y_{1,n})
\end{eqnarray*}
that $\|v_{2,1,n}\|_4^4\geq\|v_{2,1,\infty}\|_4^4+\|v_{2,2,\infty}\|_4^4+o_n(1)$.  Similarly,
\begin{eqnarray*}
\mathcal{C}_{\mathcal{N}}\geq\mathcal{C}_{\mathcal{N}_{1,2}}+\mathcal{E}_3(w_3)+o_{\beta_0}(1)+\frac{\mu_2}{4}\|w_2\|_4^4+o_n(1).
\end{eqnarray*}
It is also impossible for sufficiently small $\beta_0>0$, owing to Lemma~\ref{lem0001}.  Thus, the case~$(1)$ can not happen.  Similarly, we can show that the case~$(2)$ can not happen either, which implies the case~$(3)$ must happen.  Now, by a standard argument, $\mathcal{C}_{\mathcal{N}}$ is attained by $\overrightarrow{\hat{v}}_{1,\infty}=(|v_{1,1,\infty}|,|v_{2,1,\infty}|,|v_{3,1,\infty}|)$.  Thus, by the Harnack inequality and  Lemma~\ref{lemn0001}, \eqref{eqnew0001} has a ground state with the Morse index 3.
\end{proof}

\subsection{Ground states with the Morse index 2}
In this section, we shall study the existence of the ground states of \eqref{eqnew0001} with the Morse index 2, in the total-mixed case~$(d)$: $\beta_{1,2}>0$, $\beta_{1,3}>0$ and $\beta_{2,3}<0$.  Let
\begin{eqnarray*}
\mathcal{M}_{12,3}=\{\overrightarrow{u}\in\widehat{\mathcal{H}}_{12,3}\mid \overrightarrow{\mathcal{\widehat{Q}}}_{12,3}(u)=(\mathcal{G}_1(\overrightarrow{u})+\mathcal{G}_2(\overrightarrow{u}), \mathcal{G}_3(\overrightarrow{u}))=\overrightarrow{0}\},
\end{eqnarray*}
where $\mathcal{G}_j(\overrightarrow{u})=\|u_j\|_{\lambda_j}^2-\mu_j\|u_j\|_4^4-\sum_{i=1,i\not=j}^3\beta_{i,j}\|u_iu_j\|_2^2$ and $\widehat{\mathcal{H}}_{12,3}=((\mathcal{H}_1\times\mathcal{H}_2)\backslash\{\overrightarrow{0}\})\times(\mathcal{H}_3\backslash\{0\})$.
Let
\begin{eqnarray*}
\mathcal{C}_{\mathcal{M}_{12,3}}=\inf_{\mathcal{M}_{12,3}}\mathcal{E}(\overrightarrow{u}).
\end{eqnarray*}
Then, $\mathcal{C}_{\mathcal{M}_{12,3}}$ is well defined and nonnegative.  Using $(0,w_2,w_{3,R})$ as a test function and calculating similarly in the proof of \cite[Theorem~1]{LW05} yields
\begin{eqnarray}\label{eqnew0097}
\mathcal{C}_{\mathcal{M}_{12,3}}\leq\sum_{j=2}^3\mathcal{E}_j(w_j).
\end{eqnarray}
\begin{lemma}\label{lemn0002}
There exists $\beta_0>0$ such that $\mathcal{M}_{12,3}$ contains a $(PS)$ sequence at the least energy value $\mathcal{C}_{\mathcal{M}_{12,3}}$ for $\beta_{1,2}>0$, $0<\beta_{1,3}<\beta_0$ and $\beta_{2,3}<0$.
Moreover, any positive minimizer of $\mathcal{E}(\overrightarrow{u})$ on $\mathcal{M}_{12,3}$ is a ground state of \eqref{eqnew0001} with the Morse index 2.
\end{lemma}
\begin{proof}
The proof is similar to that of \cite[Lemma~2.1]{SW15}, so we only sketch it.  By \eqref{eqnew0097},
\begin{eqnarray*}
\mathcal{M}^*_{12,3}=\{\overrightarrow{u}\in\mathcal{M}_{12,3}\mid\sum_{j=1}^3\|u_{j}\|_{\lambda_j}^2\leq8\sum_{j=2}^3\mathcal{E}_j(w_j)\}
\not=\emptyset.
\end{eqnarray*}
Moreover, since $\beta_{2,3}<0$, there exists $\beta_0>0$ such that
\begin{eqnarray}\label{eqnew0021}
\min\{\sum_{j=1}^2\mu_j\|u_j\|_4^4+2\beta_{1,2}\|u_1u_2\|_2^2, \|u_{3}\|_{\lambda_3}^2\}\geq C_{\beta_{1,2}}>0
\end{eqnarray}
for all $\overrightarrow{u}\in\mathcal{M}^*_{12,3}$ with $\beta_{1,3}<\beta_0$, where $C_{\beta_{1,2}}$ is a constant only depending on $\beta_{1,2}$.  It follows that
\begin{eqnarray*}
\Upsilon=\left(\aligned &\sum_{j=1}^2\mu_j\|u_j\|_4^4+2\beta_{1,2}\|u_1u_2\|_2^2\quad &\sum_{j=1}^2\beta_{j,3}\|u_ju_3\|_2^2\\
&\sum_{j=1}^2\beta_{j,3}\|u_ju_3\|_2^2\quad&\mu_3\|u_3\|_4^4\endaligned\right)
\end{eqnarray*}
is strictly diagonally dominant and $|\text{det}(\Upsilon)|\geq C'_{\beta_{1,2}}>0$ for $\overrightarrow{u}\in\mathcal{M}^*_{12,3}$.
Here, $C'_{\beta_{1,2}}$ is also a constant only depending on $\beta_{1,2}$.
Now, we can follow the argument in the proof of \cite[Lemma~2.1]{SW15} to obtain a $(PS)$ sequence at the least energy value $\mathcal{C}_{\mathcal{M}_{12,3}}$ for $\beta_{1,2}>0$, $0<\beta_{1,3}<\beta_0$ and $\beta_{2,3}<0$.  For the Morse index, the proof is similar to that of Lemma~\ref{lemn0001} since we have $\mathcal{H}=\mathcal{T}_{\overrightarrow{v}}\mathcal{M}
\bigoplus(\bbr\overrightarrow{v}_{1,2}\times\bbr\overrightarrow{v}_3)$ for any positive minimizer of $\mathcal{E}(\overrightarrow{u})$ on $\mathcal{M}_{12,3}$ now, where $\overrightarrow{v}_{1,2}=(v_1,v_2,0)$ and $\overrightarrow{v}_3=(0,0,v_3)$.
\end{proof}

By Lemma~\ref{lemn0002}, to prove the existence of the ground states of \eqref{eqnew0001} with the Morse index 2, it is sufficient to prove the existence of a positive minimizer of $\mathcal{E}(\overrightarrow{u})$ on $\mathcal{M}_{12,3}$.  Let
\begin{eqnarray}
\overline{\beta}_{1,2}=\max\bigg\{\inf_{u\in H^1(\bbr^N)\backslash\{0\}}\frac{\|u\|_{\lambda_2}^2}{\|w_1u\|_2^2},
\inf_{u\in H^1(\bbr^N)\backslash\{0\}}\frac{\|u\|_{\lambda_1}^2}{\|w_2u\|_2^2}\bigg\}\label{eqn0096}
\end{eqnarray}
where $w_j$ is the unique solution of \eqref{eqnew0013}.
\begin{lemma}\label{lem0003}
Let $\beta_{1,2}>0$, $0<\beta_{1,3}<\beta_0$ and $\beta_{2,3}<0$, where $\beta_0$ is given by Lemma~\ref{lemn0002}.  If $\lambda_1<\min\{\lambda_2,\lambda_3\}$, then
\begin{eqnarray*}
\mathcal{C}_{\mathcal{M}_{12,3}}<\mathcal{C}_{\mathcal{N}_{1,2}}+\mathcal{E}_3(w_3)
\end{eqnarray*}
for $\beta_{1,2}>\overline{\beta}_{1,2}$, where $\mathcal{C}_{\mathcal{N}_{1,2}}$ is given by \eqref{eqn0110}.
\end{lemma}
\begin{proof}
Since this proof is similar to that of Lemma~\ref{lem0001}, we only sketch it and point out the differences.  By \cite[Theorems~1 and 2]{AC06}, $\mathcal{C}_{\mathcal{N}_{1,2}}$ is attained by a positive and radially symmetric function $\overrightarrow{\varphi}$ for $\beta_{1,2}>\overline{\beta}_{1,2}$.  Let $w_{3,R}=w_3(x-Re_1)$, where $e_1\in\bbr^N$ satisfying $|e_1|=1$.  We consider the following system
\begin{eqnarray*}\label{eqnew0020}
\left\{\aligned\sum_{j=1}^2\|\varphi_j\|_{\lambda_j}^2=&(\sum_{j=1}^2\mu_j\|\varphi_j\|_4^4
+2\beta_{1,2}\|\varphi_1\varphi_2\|_2^2)t^2(R)+(\sum_{j=1}^2\beta_{j,3}\|\varphi_jw_{3,R}\|_2^2)s^2(R)\\
\|w_{3}\|_{\lambda_3}^2=&(\sum_{j=1}^2\beta_{j,3}\|\varphi_jw_{3,R}\|_2^2)t^2(R)+\mu_3\|w_{3}\|_4^4s^2(R).
\endaligned\right.
\end{eqnarray*}
By Lemma~\ref{lemn0010} and $\lambda_1<\min\{\lambda_2,\lambda_3\}$, $\sum_{j=1}^2\beta_{j,3}\|\varphi_jw_{3,R}\|_2^2>0$ for $R>0$ sufficiently large.  Thus, as in the proof of Lemma~\ref{lem0001}, the above linear system is uniquely solvable for $\beta_{1,2}>\overline{\beta}_{1,2}$ and the unique solution is given by
\begin{eqnarray*}
t^2(R)=1-C(\beta_{1,3}\|\varphi_1w_{3,R}\|_2^2
+\beta_{2,3}\|\varphi_2w_{3,R}\|_2^2)
\end{eqnarray*}
and
\begin{eqnarray*}
s^2(R)=1-C'(\beta_{1,3}\|\varphi_1w_{3,R}\|_2^2
+\beta_{2,3}\|\varphi_2w_{3,R}\|_2^2)
\end{eqnarray*}
for sufficiently large $R>0$.
Moreover, $(t(R)\varphi_1, t(R)\varphi_2, s(R)w_{3,R})\in\mathcal{M}_{12,3}$.  As \eqref{eqnew9999}, applying the comparison principle yields that
\begin{eqnarray*}
\varphi_{i}(|x|)\sim |x|^{-\frac{N-1}{2}}e^{-\sqrt{\lambda_i}|x|}\quad\text{as }|x|\to+\infty.
\end{eqnarray*}
Thus, by similar estimates as that used in the proof of Lemma~\ref{lem0001}, it follows from $\lambda_1<\min\{\lambda_2, \lambda_3\}$ that
\begin{eqnarray*}
\mathcal{C}_{\mathcal{M}_{12,3}}\leq\mathcal{E}((t_1(R)\varphi_1,t_2(R)\varphi_2, t_3(R)w_{3,R}))
<\mathcal{C}_{\mathcal{N}_{1,2}}+\mathcal{E}_3(w_3),
\end{eqnarray*}
for sufficiently large $R>0$.
\end{proof}

Now, we are prepared to prove the following existence result.
\begin{proposition}\label{prop0003}
Let $\beta_{1,2}>\overline{\beta}_{1,2}$, $0<\beta_{1,3}<\beta_0$ and $\beta_{2,3}<0$, where $\overline{\beta}_{1,2}$ and $\beta_0$ are given by \eqref{eqn0096} and Lemma~\ref{lemn0002}, resectively.  If $\lambda_1<\min\{\lambda_2,\lambda_3\}$, then there exists a positive minimizer of $\mathcal{E}(\overrightarrow{u})$ on $\mathcal{M}_{12,3}$.
That is, \eqref{eqnew0001} has a ground state with the Morse index 2.
\end{proposition}
\begin{proof}
By Lemma~\ref{lemn0002}, $\mathcal{M}_{12,3}$ contains a $(PS)$ sequence of $\mathcal{E}(\overrightarrow{u})$, say $\{\overrightarrow{u}_n\}$, at the least energy value $\mathcal{C}_{\mathcal{M}_{12,3}}$.  Since $\beta_{1,2}>\overline{\beta}_{1,2}$, by \eqref{eqnew0021} and \cite[Theorems~1 and 2]{AC06}, applying the Lions lemma and the Sobolev embedding theorem in a standard way yields that, there exist $\{y_n\},\{z_n\}\subset\bbr^N$ such that $v_{j,n}=u_{j,n}(x+y_n)\rightharpoonup v_{j,\infty}\not=0$ for both $j=1,2$ and $\widehat{v}_{3,n}=u_{3,n}(x+z_n)\rightharpoonup \widehat{v}_{3,\infty}\not=0$ weakly in $H^1(\bbr^N)$ as $n\to\infty$.  Indeed, if we denote $v_{3,n}=u_{3,n}(x+y_n)$ and $\widehat{v}_{j,n}=u_{j,n}(x+z_n)$ for both $j=1,2$, then $v_{3,n}\rightharpoonup v_{3,\infty}$ and $\widehat{v}_{j,n}\rightharpoonup \widehat{v}_{j,\infty}$ weakly in $H^1(\bbr^N)$ as $n\to\infty$ for both $j=1,2$.  Now, if $\widehat{v}_{j,\infty}\not=0$ for all $j=1,2,3$, then similar as in the proof of Proposition~\ref{prop0001}, we can show that there exists a positive minimizer of $\mathcal{E}(\overrightarrow{u})$ on $\mathcal{M}_{12,3}$.  Otherwise, if
either $v_{1,\infty}=0$ or $v_{2,\infty}=0$, then by taking $\beta_0$ sufficiently small if necessary and using similar arguments in the proof of Proposition~\ref{prop0001}, $\mathcal{C}_{\mathcal{M}_{12,3}}\geq\min\{\mathcal{E}_1(w_1), \mathcal{E}_2(w_2)\}+\mathcal{E}_3(w_3)-C\beta_0$, which contradicts \cite[Theorems~1 and 2]{AC06}, $\beta_{1,2}>\overline{\beta}_{1,2}$ and Lemma~\ref{lem0003}.
We next claim that either $v_{3,\infty}\not=0$ or $\widehat{v}_{j,\infty}\not=0$ for both $j=1,2$.  Suppose the contrary; then, one of the following cases must happen:
\begin{enumerate}
\item[$(i)$]  $v_{3,\infty}=0$, $\widehat{v}_{1,\infty}=0$ and $\widehat{v}_{2,\infty}\not=0$.
\item[$(ii)$]  $v_{3,\infty}=0$, $\widehat{v}_{1,\infty}\not=0$ and $\widehat{v}_{2,\infty}=0$.
\item[$(iii)$]  $v_{3,\infty}=0$, $\widehat{v}_{1,\infty}=0$ and $\widehat{v}_{2,\infty}=0$.
\end{enumerate}
Since $\{\overrightarrow{u}_n\}$ is a $(PS)$ sequence, it is standard to show that
$$
\overrightarrow{v}_\infty=(v_{1,\infty},v_{2,\infty},v_{3,\infty})\quad\text{and}\quad
\overrightarrow{\widehat{v}}_\infty=(\widehat{v}_{1,\infty},\widehat{v}_{2,\infty},\widehat{v}_{3,\infty})
$$
are both critical points of $\mathcal{E}(\overrightarrow{u})$.  In the case~$(i)$, $(v_{1,\infty},v_{2,\infty})$ is a nontrivial critical point of $\mathcal{E}_{1,2}(\overrightarrow{\phi})$ and $(\widehat{v}_{2,\infty},\widehat{v}_{3,\infty})$ is a nontrivial critical point of $\mathcal{E}_{2,3}(\overrightarrow{\phi})$.  Since $\widehat{v}_{1,\infty}=0$ and $v_{1,\infty}\not=0$, by the Sobolev embedding theorem, $|y_n-z_n|\to+\infty$ as $n\to\infty$.  Now, as in the proof of Proposition~\ref{prop0001}, we have the following energy estimate:
\begin{eqnarray*}
\mathcal{C}_{\mathcal{M}_{12,3}}&=&\frac{1}{4}\sum_{j=1}^3\|u_{j,n}\|_{\lambda_j}^2+o_n(1)\\
&=&\frac{1}{4}\sum_{j=1}^2\|v_{j,n}\|_{\lambda_j}^2+\frac{1}{4}\|\widehat{v}_{3,n}\|_{\lambda_3}^2+o_n(1)\\
&\geq&\mathcal{C}_{\mathcal{N}_{1,2}}+\mathcal{E}_3(w_3)+\frac{1}4\|w_2\|_{\lambda_2}+o_{\beta_0}(1)+o_n(1),
\end{eqnarray*}
where $o_{\beta_0}(1)\to0$ as $\beta_0\to0$.  It contradicts Lemma~\ref{lem0003} by taking $\beta_0>0$ sufficiently small.  Thus, the case~$(i)$ is impossible.  Similarly, the case~$(ii)$ is also impossible.  It remains to exclude the case~$(iii)$.  In this case,
\begin{eqnarray*}
\mathcal{C}_{\mathcal{M}_{12,3}}&=&\frac{1}{4}\sum_{j=1}^3\|u_{j,n}\|_{\lambda_j}^2+o_n(1)\\
&=&\frac{1}{4}\sum_{j=1}^2\|v_{j,n}\|_{\lambda_j}^2
+\|\widehat{v}_{3,n}\|_{\lambda_3}^2+o_n(1)\\
&\geq&\mathcal{C}_{\mathcal{N}_{1,2}}+\mathcal{E}_3(w_3)+o_n(1),
\end{eqnarray*}
which contradicts Lemma~\ref{lem0003}.  Therefore, without loss of generality, we may assume that $(v_{1,\infty},v_{2,\infty},v_{3,\infty})$ is a nontrivial critical point of $\mathcal{E}(\overrightarrow{u})$.  By a standard argument, we can show that $\mathcal{C}_{\mathcal{M}_{12,3}}$ is attained by $(|v_{1,\infty}|,|v_{2,\infty}|,|v_{3,\infty}|)$.  By the Harnack inequality and Lemma~\ref{lemn0002}, $(|v_{1,\infty}|,|v_{2,\infty}|,|v_{3,\infty}|)$ is a ground state of \eqref{eqnew0001} with the Morse index 2.
\end{proof}

We need to further prepare an existence result for the purely attractive case: $\beta_{1,2}>0$, $\beta_{1,3}>0$ and $\beta_{2,3}>0$.  By checking the proof of Lemma~\ref{lemn0002}, we can see that it still works for $\beta_{1,2}>0$ and $0<\beta_{1,3},\beta_{2,3}<\beta_0$.  Thus, we can still work in $\mathcal{M}_{12,3}$ for $\beta_{1,2}>0$ and $0<\beta_{1,3},\beta_{2,3}<\beta_0$.  Since the Schwatz symmetrization works for this case, the minimizing sequence, at the least energy value $\mathcal{C}_{\mathcal{M}_{12,3}}$, can be chosen to be radially symmetric.  Recall that $\mathcal{C}_{\mathcal{N}_{1,2}}<\min\{\mathcal{E}_1(w_1), \mathcal{E}_2(w_2)\}$ for $\beta_{1,2}>\overline{\beta}_{1,2}$ by \cite[Theorems~1 and 2]{AC06}, by a standard argument, we can obtain the following:
\begin{proposition}\label{prop0014}
If $\beta_{1,2}>\overline{\beta}_{1,2}$ and $0<\beta_{1,3},\beta_{2,3}<\beta_0$, then there exists a positive minimizer of $\mathcal{E}(\overrightarrow{u})$ on $\mathcal{M}_{12,3}$.
That is, \eqref{eqnew0001} has a ground state with the Morse index 2.
\end{proposition}

\subsection{Ground states with the Morse index 1}
In this section, we shall study the existence of the ground states with the Morse index 1.  We define another Nehari manifold of $\mathcal{E}(\overrightarrow{u})$ as follows:
\begin{eqnarray*}
\mathcal{M}=\{\overrightarrow{u}\in\mathcal{H}\backslash\{\overrightarrow{0}\}\mid \mathcal{Q}(u)=\sum_{j=1}^3\mathcal{G}_j(\overrightarrow{u})=0\},
\end{eqnarray*}
where $\mathcal{G}_j(\overrightarrow{u})=\|u_j\|_{\lambda_j}^2-\mu_j\|u_j\|_4^4-\sum_{i=1,i\not=j}^3\beta_{i,j}\|u_iu_j\|_2^2$.
Let
\begin{eqnarray*}
\mathcal{C}_{\mathcal{M}}=\inf_{\mathcal{M}}\mathcal{E}(\overrightarrow{u}).
\end{eqnarray*}
Then, $\mathcal{C}_{\mathcal{M}}$ is well defined and nonnegative.
\begin{lemma}\label{lemn0003}
Let $\beta_{1,2}>0$, $\beta_{1,3}>0$ and $\beta_{2,3}>0$.  Then, $\mathcal{M}$ contains a $(PS)$ sequence at the least energy value $\mathcal{C}_{\mathcal{M}}$.  Moreover, any positive minimizer of $\mathcal{E}(\overrightarrow{u})$ on $\mathcal{M}$ is a ground state of \eqref{eqnew0001} with the Morse index 1.
\end{lemma}
\begin{proof}
Since $\mathcal{M}$ is homeomorphous to the set
\begin{eqnarray*}
\mathcal{O}=\{\overrightarrow{u}\in\mathcal{H}\backslash\{\overrightarrow{0}\}
\mid\sum_{j=1}^3\mu_j\|u_j\|_4^4+2\sum_{i,j=1,i<j}^3\beta_{i,j}\|u_iu_j\|_2^2>0\},
\end{eqnarray*}
the conclusion follows from a standard argument.
\end{proof}

By Lemma~\ref{lemn0003}, to prove the existence of the ground states of \eqref{eqnew0001} with the Morse index 1, it is sufficient to prove the existence of a positive minimizer of $\mathcal{E}(\overrightarrow{u})$ on $\mathcal{M}$.

Let
\begin{eqnarray*}
\mathcal{M}_{i,j}=\{\overrightarrow{\phi}\in\mathcal{H}_{i,j}\backslash\{\overrightarrow{0}\}\mid \mathcal{Q}_{i,j}(\phi)=\widehat{\mathcal{G}}_i(\overrightarrow{\phi})+\widehat{\mathcal{G}}_j(\overrightarrow{\phi})=0\},
\end{eqnarray*}
where $\mathcal{H}_{i,j}=\mathcal{H}_i\times\mathcal{H}_j$, $\widehat{\mathcal{G}}_j(\overrightarrow{\phi})=\|\phi_j\|_{\lambda_j}^2
-\mu_j\|\phi_j\|_{4}^4-\beta_{i,j}\|\phi_i\phi_j\|_2^2$ and $(i,j)$ equals to $(1,2)$, $(1,3)$ and $(2,3)$.  We define
\begin{eqnarray*}
\mathcal{C}_{\mathcal{M}_{i,j}}=\inf_{\mathcal{M}_{i,j}}\mathcal{E}_{i,j}(\overrightarrow{\phi}).
\end{eqnarray*}
Then, $\mathcal{C}_{\mathcal{M}_{i,j}}$ is well defined and nonnegative.  $\mathcal{C}_{\mathcal{M}_{i,j}}$ can also be variational expressed as follows:
\begin{eqnarray*}\label{eqn0080}
\mathcal{C}_{\mathcal{M}_{i,j}}=\inf_{\overrightarrow{u}\in(\mathcal{H}_i\times\mathcal{H}_j)\backslash\{\overrightarrow{0}\}}
\frac{(\|u_i\|_{\lambda_i}^2+\|u_j\|_{\lambda_j}^2)^2}{4(\mu_i\|u_i\|_4^4+\mu_j\|u_j\|_4^4+2\beta_{i,j}\|u_iu_j\|_2^2)}.
\end{eqnarray*}
Moreover, if $\beta_{i,j}>\overline{\beta}_{i,j}$ then $\mathcal{C}_{\mathcal{M}_{i,j}}=\mathcal{C}_{\mathcal{N}_{i,j}}$ is attained by $\overrightarrow{\varphi}^{i,j}$, which is positive and radially symmetric.  Here, $\overline{\beta}_{i,j}$ is defined as that of $\overline{\beta}_{1,2}$ at \eqref{eqn0096} (cf. \cite[Theorems~1 and 2]{AC06}).  Clearly, $\overrightarrow{\varphi}^{i,j}$ is also a solution of \eqref{eqn0099}.
\begin{lemma}\label{lem0009}
If $\beta_{i,j}\to+\infty$, then
\begin{eqnarray*}\label{eq0011}
(\sqrt{\beta_{i,j}}\varphi^{i,j}_i, \sqrt{\beta_{i,j}}\varphi^{i,j}_j)\to(\widetilde{\varphi}^{i,j}_i, \widetilde{\varphi}^{i,j}_j)
\end{eqnarray*}
up to a subsequence,
where $\overrightarrow{\widetilde{\varphi}}_{i,j}=(\widetilde{\varphi}^{i,j}_i, \widetilde{\varphi}^{i,j}_j)$, which is positive and radially symmetric, is a minimizer of the following minimizing problem:
\begin{eqnarray}\label{eqn0081}
\widetilde{\mathcal{D}}_{i,j}=\inf_{\overrightarrow{u}\in(\mathcal{H}_i\times\mathcal{H}_j)\backslash\{\overrightarrow{0}\}}
\frac{(\|u_i\|_{\lambda_i}^2+\|u_j\|_{\lambda_j}^2)^2}{8\|u_iu_j\|_2^2}.
\end{eqnarray}
\end{lemma}
\begin{proof}
Using the Schwatz symmetrization and the Sobolev embedding theorem in a standard way yields that $\widetilde{\mathcal{D}}_{i,j}$ is attained by $\overrightarrow{\check{\varphi}}_{i,j}$, which is positive and radially symmetric.  Testing $\mathcal{C}_{\mathcal{M}_{i,j}}$ by  $\overrightarrow{\check{\varphi}}_{i,j}$ yields that $\mathcal{C}_{\mathcal{M}_{i,j}}\beta_{i,j}\leq\widetilde{\mathcal{D}}_{i,j}+o(1)$ as $\beta_{i,j}\to+\infty$, where $o(1)\to0$ as $\beta_{i,j}\to+\infty$.  It follows that $(\sqrt{\beta_{i,j}}\varphi^{i,j}_i, \sqrt{\beta_{i,j}}\varphi^{i,j}_j)$ is bounded in $\mathcal{H}_i\times\mathcal{H}_j$ for $\beta_{i,j}>0$ sufficiently large.  On the other hand, it is easy to see that $\mathcal{C}_{\mathcal{M}_{i,j}}\to0$ as $\beta_{i,j}\to+\infty$.  It follows that $\|\varphi_i\|_{\lambda_i}^2+\|\varphi_j\|_{\lambda_j}^2\to0$ as $\beta_{i,j}\to+\infty$.  By the H\"older and Sobolev inequalities, $\mu_i\|\varphi_i\|_4^4+\mu_j\|\varphi_j\|_4^4=o(\|\varphi_i\|_{\lambda_i}^2+\|\varphi_j\|_{\lambda_j}^2)$ as $\beta_{i,j}\to+\infty$.  Thus, testing $\widetilde{\mathcal{D}}_{i,j}$ by $(\sqrt{\beta_{i,j}}\varphi^{i,j}_i, \sqrt{\beta_{i,j}}\varphi^{i,j}_j)$ yields that
$\widetilde{\mathcal{D}}_{i,j}\leq\mathcal{C}_{\mathcal{M}_{i,j}}\beta_{i,j}+o(1)$ as $\beta_{i,j}\to+\infty$.  Therefore, $\mathcal{C}_{\mathcal{M}_{i,j}}\beta_{i,j}=\widetilde{\mathcal{D}}_{i,j}+o(1)$ as $\beta_{i,j}\to+\infty$.  Since $\overrightarrow{\varphi}^{i,j}$ is radially symmetric, it is standard to show that
\begin{eqnarray*}
(\sqrt{\beta_{i,j}}\varphi^{i,j}_i, \sqrt{\beta_{i,j}}\varphi^{i,j}_j)\to(\widetilde{\varphi}^{i,j}_i, \widetilde{\varphi}^{i,j}_j)
\end{eqnarray*}
as $\beta_{i,j}\to+\infty$ up to a subsequence, where $\overrightarrow{\widetilde{\varphi}}_{i,j}$, which is positive and radially symmetric, is a minimizer of \eqref{eqn0081}.
\end{proof}

Let
\begin{eqnarray}\label{eqn0078}
\rho_{ij,l}=\inf_{u\in H^1(\bbr^N)\backslash\{0\}}\frac{\|u\|_{\lambda_l}^2}{\int_{\bbr^N}((\varphi^{i,j}_i)^2+(\varphi^{i,j}_j)^2)u^2dx},
\end{eqnarray}
where $i,j,l=1,2,3$ with $i\not=j$, $i\not=l$ and $j\not=l$.
It follows from Lemma~\ref{lem0009} that
\begin{eqnarray}\label{eqn0077}
\rho_{ij,l}=\beta_{i,j}(\widehat{\rho}_{ij,l}+o(1))\text{ as }\beta_{i,j}\to+\infty,
\end{eqnarray}
where
\begin{eqnarray}\label{eqnew3344}
\widehat{\rho}_{ij,l}=\inf_{u\in H^1(\bbr^N)\backslash\{0\}}\frac{\|u\|_{\lambda_l}^2}{\int_{\bbr^N}((\widetilde{\varphi}^{i,j}_i)^2+(\widetilde{\varphi}^{i,j}_j)^2)u^2dx}.
\end{eqnarray}
Since $\overrightarrow{\varphi}^{i,j}$ is  a solution of \eqref{eqn0099}, by Lemma~\ref{lem0009}, $\overrightarrow{\widetilde{\varphi}}_{i,j}$ also satisfies the following system:
\begin{equation}\label{eqnewnew0009}
\left\{\aligned&-\Delta\widetilde{\varphi}^{i,j}_i+\lambda_i\widetilde{\varphi}^{i,j}_i
=(\widetilde{\varphi}^{i,j}_j)^2\widetilde{\varphi}^{i,j}_i\quad\text{in }\bbr^N,\\
&-\Delta\widetilde{\varphi}^{i,j}_j+\lambda_j\widetilde{\varphi}^{i,j}_j
=(\widetilde{\varphi}^{i,j}_i)^2\widetilde{\varphi}^{i,j}_j\quad\text{in }\bbr^N.
\endaligned\right.
\end{equation}
\begin{proposition}\label{prop0005}
Let $\beta_{1,2}>0$, $\beta_{1,3}>0$ and $\beta_{2,3}>0$.  Then, there exist $\widehat{\beta}_0>0$  such that if $\min\{\beta_{i,j}\}>\widehat{\beta}_0$ and
\begin{eqnarray}\label{eqn0079}
\widehat{\rho}_{jl,i}<\frac{\beta_{i,l}}{\beta_{j,l}}<\frac{1}{\widehat{\rho}_{il,j}}
\end{eqnarray}
for all $i,j,l=1,2,3$ with $i\not=j$, $i\not=l$ and $l\not=j$, then $\mathcal{C}_{\mathcal{M}}<\min\{\mathcal{C}_{\mathcal{M}_{i,j}}\}$ and consequently there exists a positive minimizer of $\mathcal{E}(\overrightarrow{u})$ on $\mathcal{M}$, provided that $|\lambda_i-\lambda_j|<<1$ for all $i,j=1,2,3$ with $i\not=j$.  That is, \eqref{eqnew0001} has a ground state with the Morse index 1.
\end{proposition}
\begin{proof}
Let us first prove that $\widehat{\rho}_{jl,i}<\frac{1}{\widehat{\rho}_{il,j}}$ for all $i,j,l=1,2,3$ with $i\not=j$, $i\not=l$ and $l\not=j$, provided that $|\lambda_i-\lambda_j|<<1$ for all $i,j=1,2,3$ with $i\not=j$.  Without loss of generality, we assume that $\lambda_1\leq\lambda_2\leq\lambda_3$.  Testing $\widehat{\rho}_{13,2}$ by $\widetilde{\varphi}^{1,3}_3$ yields that
\begin{eqnarray*}
\widehat{\rho}_{13,2}\leq\frac{\|\widetilde{\varphi}^{1,3}_3\|_{\lambda_2}^2}{\|\widetilde{\varphi}^{1,3}_1\widetilde{\varphi}^{1,3}_3\|_2^2
+\|\widetilde{\varphi}^{1,3}_3\|_4^4}
<\frac{\|\widetilde{\varphi}^{1,3}_3\|_{\lambda_3}^2+(\lambda_2-\lambda_3)\|\widetilde{\varphi}^{1,3}_3\|_{2}^2}
{\|\widetilde{\varphi}^{1,3}_1\widetilde{\varphi}^{1,3}_3\|_2^2}
\leq1.
\end{eqnarray*}
Similarly, testing $\widehat{\rho}_{23,1}$ by $\widetilde{\varphi}^{2,3}_3$ yields that $\widehat{\rho}_{23,1}<1$.  For $\widehat{\rho}_{12,3}$,  by the Pohozaev identity,
\begin{eqnarray}\label{eqnewnew0010}
\lambda_j\|w_{j}\|_2^2=\frac{(4-N)\mu_j}{4}\|w_{j}\|_4^4,
\end{eqnarray}
where $w_j$ is the unique solution of \eqref{eqnew0013}.  On the other hand, it is well known that $\inf_{u\in H^1(\bbr^N)\backslash\{0\}}\frac{\|u\|_{\lambda_j}^2}{\|u\|_4^2}=\mu_j\|w_{j}\|_4^2$.  Thus, by \eqref{eqnewnew0009}, $\mu_1\|w_{1}\|_4^2\leq\|\widetilde{\varphi}^{1,2}_2\|_4^2$.  Now, testing $\widehat{\rho}_{12,3}$ by $\widetilde{\varphi}^{1,2}_2$,
\begin{eqnarray*}
\widehat{\rho}_{12,3}\leq\frac{\|\widetilde{\varphi}^{1,2}_2\|_{\lambda_3}^2}{\|\widetilde{\varphi}^{1,2}_1\widetilde{\varphi}^{1,2}_2\|_2^2
+\|\widetilde{\varphi}^{1,2}_2\|_4^4}
\leq1+\frac{(\lambda_3-\lambda_2)\|\widetilde{\varphi}^{1,2}_2\|_{2}^2-\mu_1^2\|w_{1}\|_4^4}
{\|\widetilde{\varphi}^{1,2}_1\widetilde{\varphi}^{1,2}_2\|_2^2+\mu_1^2\|w_{1}\|_4^4}.
\end{eqnarray*}
Since $\|\widetilde{\varphi}^{1,2}_2\|_{2}^2\leq\frac{4}{\lambda_2}\widetilde{\mathcal{D}}_{1,2}$, testing $\widetilde{\mathcal{D}}_{1,2}$ by $(w_1,w_1)$ and using \eqref{eqnewnew0010} yields that
\begin{eqnarray*}
\widetilde{\mathcal{D}}_{1,2}&\leq&\frac{(\|w_{1}\|_{\lambda_1}^2+\|w_{1}\|_{\lambda_2}^2)^2}{8\|w_{1}\|_4^4}\leq\mu_1^2(\frac12
+\frac{C(\lambda_2-\lambda_1)}{\lambda_1}+\frac{C'(\lambda_2-\lambda_1)^2}{\lambda_1^2})\|w_{1}\|_4^4.
\end{eqnarray*}
Thus, there exists $\delta_0>0$, only depending on $\min\{\lambda_i\}$, such that if $|\lambda_i-\lambda_j|\leq\delta_0$, then
$\widehat{\rho}_{jl,i}<\frac{1}{\widehat{\rho}_{il,j}}$ for all $i,j,l=1,2,3$ with $i\not=j$, $i\not=l$ and $l\not=j$.  It follows that there exists $\widehat{\beta}_0>0$ such that \eqref{eqn0079} holds for $\beta_{i,j}<\widehat{\beta}_0$ and $|\beta_{i,j}-\beta_{i,l}|<<1$ for all $i,j,l=1,2,3$ with $i\not=j$, $i\not=l$ and $l\not=j$.  Since $\mathcal{C}_{\mathcal{M}}$ can also be variational expressed as follows:
\begin{eqnarray*}
\mathcal{C}_{\mathcal{M}}=\inf_{\overrightarrow{u}\in\mathcal{H}\backslash\{\overrightarrow{0}\}}
\frac{(\sum_{j=1}^3\|u_j\|_{\lambda_j}^2)^2}
{4(\sum_{j=1}^3\mu_j\|u_j\|_4^4+2\sum_{i,j=1,i<j}^3\beta_{i,j}\|u_iu_j\|_2^2)},
\end{eqnarray*}
testing $\mathcal{C}_{\mathcal{M}}$ by $\overrightarrow{V}_s=(\varphi^{1,2}_1,\varphi^{1,2}_2,su)$ yields that
\begin{eqnarray}
\mathcal{C}_{\mathcal{M}}\leq\mathcal{C}_{\mathcal{M}_{1,2}}+
\frac{s^2}{2}(\|u\|_{\lambda_3}^2-\sum_{j=1}^2\beta_{j,3}\|\varphi^{1,2}_ju\|_2^2)+O(s^4).\label{eq0009}
\end{eqnarray}
Let $u=\psi_{12,3}$ be the minimizer of \eqref{eqn0078}.  Then, by \eqref{eqn0077}, \eqref{eqn0079} and \eqref{eq0009},
\begin{eqnarray*}
\mathcal{C}_{\mathcal{M}}&\leq&\mathcal{C}_{\mathcal{M}_{1,2}}+\frac{s^2}{2}(\rho_{12,3}\sum_{j=1}^2\|\varphi^{1,2}_j\psi_{12,3}\|_2^2
-\sum_{j=1}^2\beta_{j,3}\|\varphi^{1,2}_j\psi_{12,3}\|_2^2)+O(s^4)\\
&=&\mathcal{C}_{\mathcal{M}_{1,2}}+\frac{s^2}{2}(\beta_{1,2}\widehat{\rho}_{12,3}\sum_{j=1}^2\|\varphi^{1,2}_j\psi_{12,3}\|_2^2
-\sum_{j=1}^2\beta_{j,3}\|\varphi^{1,2}_j\psi_{12,3}\|_2^2)\\
&&+o(s^2)\\
&<&\mathcal{C}_{\mathcal{M}_{1,2}}
\end{eqnarray*}
for $s>0$ sufficiently small by taking $\widehat{\beta}_0>0$ sufficiently large.  Similarly,
\begin{eqnarray*}
\mathcal{C}_{\mathcal{M}}<\mathcal{C}_{\mathcal{M}_{1,3}}\quad\text{ for }\beta_{1,3}>\widehat{\beta}_0\quad\text{and}\quad
\mathcal{C}_{\mathcal{M}}<\mathcal{C}_{\mathcal{M}_{2,3}}\quad\text{ for }\beta_{2,3}>\widehat{\beta}_0.
\end{eqnarray*}
Since we have already shown that $\mathcal{C}_{\mathcal{M}}<\min\{\mathcal{C}_{\mathcal{M}_{i,j}}\}$ for $\min\{\beta_{i,j}\}>\widehat{\beta}_0>0$, it is standard to use the Schwatz symmetrization to show that there exists a positive minimizer of $\mathcal{E}(\overrightarrow{u})$ on $\mathcal{M}$, which implies that \eqref{eqnew0001} has a ground state with the Morse index 1.
\end{proof}

\subsection{Nonexistence of ground states}
In this section, let us focus our attention on the nonexistence of the ground states of \eqref{eqnew0001}, in the total-mixed case~$(d)$: $\beta_{1,2}>0$, $\beta_{1,3}>0$ and $\beta_{2,3}<0$.  We begin with the following observation.
\begin{lemma}\label{lem0002}
Let $\beta_{1,2}=\delta\widehat{\beta}_{1,2}$, $\beta_{1,3}=\delta^t\widehat{\beta}_{1,3}$ and $\beta_{2,3}=-\delta^s\widehat{\beta}_{2,3}$, where $\delta>0$ is a parameter, $0<s<\min\{1,t\}$ and $\widehat{\beta}_{i,j}$ are positively absolute constants.  Suppose that $\overrightarrow{u}_\delta$ is a ground state of \eqref{eqnew0001} and $y_{j,\delta}$ is the maximum point of $u_{j,\delta}$, respectively.  Then, $\widehat{v}_{j,\delta}=u_{j,\delta}(x+y_{j,\delta})\to w_j$ strongly in $H^1(\bbr^N)\cap L^\infty(\bbr^N)$ as $\delta\to0$ up to a subsequence.  Moreover, either
\begin{enumerate}
\item[$(i)$]  $y_{1,\delta}-y_{2,\delta}\to0$ and $|y_{2,\delta}-y_{3,\delta}|\to+\infty$ or
\item[$(ii)$]  $y_{1,\delta}-y_{3,\delta}\to0$ and $|y_{2,\delta}-y_{3,\delta}|\to+\infty$.
\end{enumerate}

\end{lemma}
\begin{proof}
We respectively re-denote $\mathcal{C}_{\mathcal{N}}$ and $\mathcal{C}_{\mathcal{N}_{i,j}}$ by $\mathcal{C}_{\mathcal{N}}^\delta$ and $\mathcal{C}_{\mathcal{N}_{i,j}}^\delta$ for the sake of clarity in this proof, where $\mathcal{C}_{\mathcal{N}_{i,j}}$ is given by \eqref{eqn0110} and $(i,j)$ equals to $(1,2)$, $(1,3)$ or $(2,3)$.  We also re-denote $\overrightarrow{\varphi}^{i,j}$ by $\overrightarrow{\varphi}_\delta^{i,j}$, where $\overrightarrow{\varphi}^{i,j}=(\varphi^{i,j}_i,\varphi^{i,j}_j)$ is a ground state of \eqref{eqn0099} and $(i,j)$ equals to $(1,2)$ or $(1,3)$.  As in the proof of Lemma~\ref{lem0001}, Using $(\varphi_{1,\delta}^{1,2},\varphi_{2,\delta}^{1,2},w_{3,R})$ as a test function of $\mathcal{C}_{\mathcal{N}}^\delta$ and letting $R\to+\infty$ yields that
\begin{eqnarray}\label{eqnew0002}
\mathcal{C}_{\mathcal{N}}^\delta\leq\mathcal{C}_{\mathcal{N}_{1,2}}^\delta+\mathcal{E}_3(w_3),
\end{eqnarray}
which together with Remark~\ref{rmk0001}, implies $\mathcal{C}_{\mathcal{N}}^\delta\leq\sum_{j=1}^3\mathcal{E}_j(w_j)-C\delta$ for sufficiently small $\delta>0$.  Similarly, if we test $\mathcal{C}_{\mathcal{N}}^\delta$ by $(\varphi_{1,\delta}^{1,3},w_{2,R},\varphi_{3,\delta}^{1,3})$, then we obtain $\mathcal{C}_{\mathcal{N}}^\delta\leq\sum_{j=1}^3\mathcal{E}_j(w_j)-C\delta^t$ for sufficiently small $\delta>0$.  Hence, we always have
\begin{eqnarray}\label{eq0004}
\mathcal{C}_{\mathcal{N}}^\delta\leq\sum_{j=1}^3\mathcal{E}_j(w_j)-C\delta^{\min\{1,t\}}\quad\text{for sufficiently small }\delta>0.
\end{eqnarray}
On the other hand, applying the Lions lemma and the Sobolev embedding theorem in a standard way yields that there exist $\{z_{j,\delta}\}\subset\bbr^N$ such that $\widehat{v}_{j,\delta}=u_{j,\delta}(x+z_{j,\delta})\to w_j$ strongly in $H^1(\bbr^N)$ as $\delta\to0$ up to a subsequence.  Let $v_{j,\delta}=\widehat{v}_{j,\delta}-w_j$, then $v_{j,\delta}$ satisfies the following equation
\begin{eqnarray}
-\Delta v_{j,\delta}+\lambda_jv_{j,\delta}&=&\mu_j[3w_j^2v_{j,\delta}+3w_j(v_{j,\delta})^2+(v_{j,\delta})^3]\notag\\
&&+\beta_{i,j}(\widehat{v}_{i,\delta})^2\widehat{v}_{j,\delta}+\beta_{l,j}(\widehat{v}_{l,\delta})^2\widehat{v}_{j,\delta}\label{eqnew9995}
\end{eqnarray}
in $\bbr^N$,
where $i,j,l=1,2,3$ with $i\not=l$, $l\not=j$ and $i\not=j$.
Applying the Moser iteration in a standard way yields that $v_{j,\delta}\to0$ strongly in $L^p(\bbr^N)$ for all $p\geq2$ as $\delta\to0$ up to a subsequence.  Using the classical elliptic estimates in a standard way yields that $\widehat{v}_{j,\delta}\to w_j$ strongly in $L^\infty(\bbr^N)$ as $\delta\to0$ up to a subsequence.  In particular, $|\widehat{v}_{j,\delta}(x)|<<1$ for $|x|>>1$ uniformly for sufficiently small $\delta>0$.  Since $y_{j,\delta}$ is the maximum point of $u_{j,\delta}$, $|y_{j,\delta}-z_{j,\delta}|\lesssim1$ for sufficiently small $\delta>0$.  Thus, since $w_j(0)=\max_{x\in\bbr^N}w_j(x)$ and $y_{j,\delta}$ is the maximum point of $u_{j,\delta}$, we may assume that $z_{j,\delta}=y_{j,\delta}$ for sufficiently small $\delta>0$.  That is, $\widehat{v}_{j,\delta}=u_{j,\delta}(x+y_{j,\delta})\to w_j$ strongly in $H^1(\bbr^N)\cap L^\infty(\bbr^N)$ as $\delta\to0$ up to a subsequence.
Since by scaling, the best embedding constant from $\mathcal{H}_j$ to $L^4(\bbr^N)$ is $\mu_j\|w_j\|_4^2$, $\|u_{j,\delta}\|_{\lambda_j}^2\geq\mu_j\|w_j\|_4^2\|u_{j,\delta}\|_4^2$.  It follows that
\begin{eqnarray*}
\mu_j\|u_{j,\delta}\|_4^2\geq\mu_j\|w_j\|_4^2-\frac{1}{\|u_{j,\delta}\|_4^2}(\beta_{i,j}\|u_{i,\delta}u_{j,\delta}\|_2^2
+\beta_{l,j}\|u_{l,\delta}u_{j,\delta}\|_2^2),
\end{eqnarray*}
which implies
\begin{eqnarray}\label{eqnewnew0003}
\|u_{j,\delta}\|_{\lambda_j}^2\geq\mu_j\|w_j\|_4^4-\frac{\|w_j\|_4^2}{\|u_{j,\delta}\|_4^2}(\beta_{i,j}\|u_{i,\delta}u_{j,\delta}\|_2^2
+\beta_{l,j}\|u_{l,\delta}u_{j,\delta}\|_2^2).
\end{eqnarray}
Here, $i,j,l=1,2,3$ with $i\not=l$, $l\not=j$ and $i\not=j$.  Therefore, we have a lower-bound estimate of $\mathcal{C}_{\mathcal{N}}^\delta$ as follows:
\begin{eqnarray}
\mathcal{C}_{\mathcal{N}}^\delta&\geq&\sum_{j=1}^3\mathcal{E}_j(w_j)
-\frac{1+o_\delta(1)}{2}\sum_{i,j=1,i<j}^3\beta_{i,j}\|u_{i,\delta}u_{j,\delta}\|_2^2\label{eqnew9997}\\
&\geq&\sum_{j=1}^3\mathcal{E}_j(w_j)-\frac{C+o_\delta(1)}{2}(\|u_{1,\delta}u_{2,\delta}\|_2^2\delta+\|u_{1,\delta}u_{3,\delta}\|_2^2\delta^t)\notag\\
&&+\frac{C'+o_\delta(1)}{2}\|u_{2,\delta}u_{3,\delta}\|_2^2\delta^s.\label{eq0005}
\end{eqnarray}
Here, $o_\delta(1)\to0$ as $\delta\to0$.
If both $\|u_{1,\delta}u_{2,\delta}\|_2^2$ and $\|u_{1,\delta}u_{3,\delta}\|_2^2$ converge to $0$ as $\delta\to0$ or $1\lesssim\|u_{2,\delta}u_{3,\delta}\|_2^2$ for sufficiently small $\delta>0$, then \eqref{eq0004} and \eqref{eq0005} can not hold at the same time for sufficiently small $\delta>0$.  Thus, either
\begin{enumerate}
\item[$(1)$]  $1\lesssim\|u_{1,\delta}u_{2,\delta}\|_2^2$ and $\|u_{2,\delta}u_{3,\delta}\|_2^2=o_\delta(1)$ or
\item[$(2)$]  $1\lesssim\|u_{1,\delta}u_{3,\delta}\|_2^2$ and $\|u_{2,\delta}u_{3,\delta}\|_2^2=o_\delta(1)$
\end{enumerate}
as $\delta\to0$.  By the Lebesgue dominated convergence theorem, either
\begin{enumerate}
\item[$(i)$]  $|y_{1,\delta}-y_{2,\delta}|\lesssim1$ and $|y_{2,\delta}-y_{3,\delta}|\to+\infty$ or
\item[$(ii)$]  $|y_{1,\delta}-y_{3,\delta}|\lesssim1$ and $|y_{2,\delta}-y_{3,\delta}|\to+\infty$
\end{enumerate}
as $\delta\to0$.  Without loss of generality, we assume $y_{1,\delta}-y_{2,\delta}\to y_0$ as $\delta\to0$ in the case~$(i)$ and $y_{1,\delta}-y_{3,\delta}\to y_0'$ as $\delta\to0$ in the case~$(ii)$.  It remains to show that both $y_0$ and $y_0'$ equal to $0$.  In what follows, we only give the proof of $y_0$ since that of $y_0'$ is similar.  In the case~$(i)$, we also have $|y_{1,\delta}-y_{3,\delta}|\to+\infty$ as $\delta\to0$ and $t\geq1$.  It follows from the Lebesgue dominated convergence theorem that $\|u_{1,\delta}u_{3,\delta}\|_2^2=o_\delta(1)$ and
$$
\|u_{1,\delta}u_{2,\delta}\|_2^2=\int_{\bbr^N}w_1(x)^2w_2(x+y_0)^2dx+o_\delta(1).
$$
Moreover, since $\widehat{v}_{j,\delta}=u_{j,\delta}(x+y_{j,\delta})\to w_j$ strongly in $H^1(\bbr^N)\cap L^\infty(\bbr^N)$ as $\delta\to0$ up to a subsequence, it is standard to show that there exist $t_j(\delta)\to1$ and $s(\delta)\to1$ as $\delta\to0$ such that $(t_1(\delta)u_{1,\delta}, t_2(\delta)u_{2,\delta})\in\mathcal{N}_{1,2}$ and $s(\delta)u_{3,\delta}\in\mathcal{N}_{3}$.  Thus, by \cite[Theorem~5]{LW05},
\begin{eqnarray}
\mathcal{C}_{\mathcal{N}}&=&\mathcal{E}(\overrightarrow{u}_\delta)\notag\\
&\geq&\mathcal{E}((t_1(\delta)u_{1,\delta}, t_2(\delta)u_{2,\delta},s(\delta)u_{3,\delta}))\notag\\
&\geq&\mathcal{C}_{\mathcal{N}_{1,2}}+\mathcal{E}_3(w_3)\notag\\
&&-\frac{1+o_\delta(1)}{2}(\beta_{1,3}\|u_{1,\delta}u_{3,\delta}\|_2^2+\beta_{2,3}\|u_{2,\delta}u_{3,\delta}\|_2^2),\label{eqnew9987}
\end{eqnarray}
which together with \eqref{eqnew0002}, implies $\beta_{1,3}\|u_{1,\delta}u_{3,\delta}\|_2^2+\beta_{2,3}\|u_{2,\delta}u_{3,\delta}\|_2^2\geq0$.  Thus,
by Remark~\ref{rmk0001} and \eqref{eqnew9997},
$$
\|u_{1,\delta}u_{2,\delta}\|_2^2\geq\max\{\|w_1w_2\|_2^2, \|w_1w_3\|_2^2\}+o_\delta(1).
$$
It follows that
\begin{eqnarray}\label{eqnew9996}
\int_{\bbr^N}w_1(x)^2(w_2(x+y_0))^2dx\geq\int_{\bbr^N}w_1(x)^2w_2(x)^2dx.
\end{eqnarray}
Let $F(z)=\int_{\bbr^N}w_1(x)^2(w_2(x+z))^2dx$.  Then,
\begin{eqnarray}
\nabla F(z)&=&\int_{\bbr^N}2w_1(|x|)^2w_2(|x+z|)w_2'(|x+z|)\frac{x+z}{|x+z|}dx\notag\\
&=&\int_{\bbr^N}2w_1(|x-z|)^2w_2(|x|)w_2'(|x|)\frac{x}{|x|}dx.\label{eqnew6666}
\end{eqnarray}
Since $w_1(x)$ and $w_2(x)$ are radially symmetric and strictly decreasing for $|x|$, $\nabla F(z)=0$ if and only if $z=0$.  Thus, by $F(z)>0$ and $F(z)\to0$ as $|z|\to+\infty$, $F(0)=\max_{z\in\bbr^N}F(z)$.  It follows from \eqref{eqnew9996} that $y_0=0$.
\end{proof}

Now, we are prepared to prove the following nonexistence result.
\begin{proposition}\label{prop0002}
Let $\beta_{1,2}=\delta\widehat{\beta}_{1,2}$, $\beta_{1,3}=\delta^t\widehat{\beta}_{1,3}$ and $\beta_{2,3}=-\delta^s\widehat{\beta}_{2,3}$, where $\delta>0$ is a parameter, $0<s<\min\{1,t\}$ and $\widehat{\beta}_{i,j}$ are positively absolute constants.  If $\lambda_1\geq\min\{\lambda_2,\lambda_3\}$ then $\mathcal{C}_{\mathcal{N}}$ can not be attained for sufficiently small $\delta>0$.  That is, \eqref{eqnew0001} has no ground states.
\end{proposition}
\begin{proof}
Let us assume the contrary that \eqref{eqnew0001} has a ground state $\overrightarrow{u}_\delta$ for sufficiently small $\delta>0$, in the case $\lambda_1\geq\min\{\lambda_2,\lambda_3\}$.  Let $y_{j,\delta}$ be the maximum point of $u_{j,\delta}$, respectively.  Then, by Lemma~\ref{lem0002}, $\widehat{v}_{j,\delta}=u_{j,\delta}(x+y_{j,\delta})\to w_j$ strongly in $H^1(\bbr^N)\cap L^\infty(\bbr^N)$ as $\delta\to0$ up to a subsequence.  Moreover, either
\begin{enumerate}
\item[$(i)$]  $y_{1,\delta}-y_{2,\delta}\to0$ and $|y_{2,\delta}-y_{3,\delta}|\to+\infty$ or
\item[$(ii)$]  $y_{1,\delta}-y_{3,\delta}\to0$ and $|y_{2,\delta}-y_{3,\delta}|\to+\infty$.
\end{enumerate}
Without loss of generality, we assume that the case~$(i)$ happens.  Let $\{\alpha_{j,l}\}_{l=0,1,2,\cdots}$ be the eigenvalues of the following eigenvalue problem:
\begin{eqnarray*}
-\Delta v+\lambda_j v=\alpha w_j^2v,\quad v\in H^1(\bbr^N).
\end{eqnarray*}
Then, it is well-known that $\alpha_{j,0}=1$, $\alpha_{j,1}=\alpha_{j,2}=\cdots=\alpha_{j,N}=3$ and $\alpha_{j,l}>3$ for $l=N+1,N+2,\cdots$.  Let $\upsilon_{j,l}$ be the corresponding eigenfunction of $\alpha_{j,l}$.  Then, it is also well-known that $H^1(\bbr^N)=\bigoplus_{l=1}^\infty\bbr\upsilon_{j,l}$ and $\upsilon_{j,0}=w_j$ and $\upsilon_{j,l}=\frac{\partial w_j}{\partial x_l}$ for $l=1,2,\cdots,N$.
Moreover, $|\upsilon_{j,n}(x)|\lesssim|\upsilon_{j,0}(x)|$ for $|x|$ sufficiently large.  Since $\widehat{v}_{j,\delta}\to w_j$ strongly in $H^1(\bbr^N)\cap L^\infty(\bbr^N)$ as $\delta\to0$ up to a subsequence, $v_{j,\delta}=\sum_{l=1}^\infty\gamma_{j,l}^\delta\upsilon_{j,l}$ with $\gamma_{j,l}^\delta\to0$ as $\delta\to0$, where $v_{j,\delta}$ is the solution of \eqref{eqnew9995}.  Thus,
\begin{eqnarray*}
\bigg(\frac{v_{j,\delta}}{w_j}\bigg)^2\lesssim\sum_{l=1}^{\infty}(\gamma_{j,l}^\delta)^2=\int_{\bbr^N}w_j^2v_{j,\delta}^2dx=o_{\delta}(1).
\end{eqnarray*}
Here, without loss of generality, we assume that $\int_{\bbr^N}w_j^2\upsilon_{j,l}^2dx=1$ for all $j=1,2,3$ and $l=1,2,3,\cdots$.
For the sake of simplicity, we assume $y_{1,\delta}=0$ and denote $w_{j,y}=w_j(x+y)$ in what follows.  Thus,
\begin{eqnarray}
&&\int_{\bbr^N}(u_{1,\delta})^2(u_{3,\delta})^2dx\notag\\
&=&\int_{\bbr^N}w_1^2w_{3,-y_{3,\delta}}^2dx
+2\int_{\bbr^N}w_1^2w_{3,-y_{3,\delta}}v_{3,\delta}(x-y_{3,\delta})dx\notag\\
&&+\int_{\bbr^N}w_1^2(v_{3,\delta}(x-y_{3,\delta}))^2dx+2\int_{\bbr^N}w_{3,-y_{3,\delta}}^2 w_1v_{1,\delta}dx\notag\\
&&+4\int_{\bbr^N}w_1v_{1,\delta}w_{3,-y_{3,\delta}}v_{3,\delta}(x-y_{3,\delta})dx\notag\\
&&+2\int_{\bbr^N}w_1v_{1,\delta}(v_{3,\delta}(x-y_{3,\delta}))^2dx+\int_{\bbr^N}v_{1,\delta}^2w_{3,-y_{3,\delta}}^2dx\notag\\
&&+2\int_{\bbr^N}v_{1,\delta}^2w_{3,-y_{3,\delta}}v_{3,\delta}(x-y_{3,\delta})dx+\int_{\bbr^N}v_{1,\delta}^2(v_{3,\delta}(x-y_{3,\delta}))^2dx
\notag\\
&=&(1+o_\delta(1))\int_{\bbr^N}w_1^2w_{3,-y_{3,\delta}}^2dx.\label{eqnew9994}
\end{eqnarray}
Similarly,
\begin{eqnarray}
\int_{\bbr^N}(u_{2,\delta})^2(u_{3,\delta})^2dx&=&(1+o_\delta(1))\int_{\bbr^N}w_2^2w_{3,-y_{3,\delta}}^2dx.\label{eqnew9993}
\end{eqnarray}
Since $|y_{3,\delta}|\to+\infty$ as $\delta\to0$, by Lemma~\ref{lemn0010},
\begin{eqnarray*}
\int_{\bbr^N}w_1^2w_{3,-y_{3,\delta}}^2dx\sim\left\{\aligned |y_{3,\delta}|^{1-N}e^{-2\min\{\sqrt{\lambda_1}, \sqrt{\lambda_3}\}|y_{3,\delta}|},\quad \lambda_1\not=\lambda_3;\\
|y_{3,\delta}|^{1+\alpha-N}e^{-2\sqrt{\lambda}|y_{3,\delta}|},\quad \lambda_1=\lambda_3=\lambda\endaligned
\right.
\end{eqnarray*}
and
\begin{eqnarray*}
\int_{\bbr^N}w_2^2w_{3,-y_{3,\delta}}^2dx\sim\left\{\aligned |y_{3,\delta}|^{1-N}e^{-2\min\{\sqrt{\lambda_2}, \sqrt{\lambda_3}\}|y_{3,\delta}|},\quad \lambda_2\not=\lambda_3;\\
|y_{3,\delta}|^{1+\alpha-N}e^{-2\sqrt{\lambda}|y_{3,\delta}|},\quad \lambda_2=\lambda_3=\lambda\endaligned
\right.
\end{eqnarray*}
as $\delta\to0$, where $\alpha=1$ for $N=1$ and $\alpha=\frac{1}{2}$ for $N=2,3$.  Since $s<t$ and $\lambda_1\geq\min\{\lambda_2,\lambda_3\}$, it follows from \eqref{eqnew9994} and \eqref{eqnew9993} that
\begin{eqnarray}
\beta_{1,3}\|u_{1,\delta}u_{3,\delta}\|_2^2+\beta_{2,3}\|u_{2,\delta}u_{3,\delta}\|_2^2&\leq&
\delta^t\widehat{\beta}_{1,3}(1+o_\delta(1))\int_{\bbr^N}w_1^2w_{3,-y_{3,\delta}}^2dx\notag\\
&&-\delta^s\widehat{\beta}_{2,3}(1+o_\delta(1))\int_{\bbr^N}w_2^2w_{3,-y_{3,\delta}}^2dx\notag\\
&<0&\label{eqnew9991}
\end{eqnarray}
for sufficiently small $\delta>0$.
On the other hand, since $\widehat{v}_{j,\delta}=u_{j,\delta}(x+y_{j,\delta})\to w_j$ strongly in $H^1(\bbr^N)$ as $\delta\to0$ up to a subsequence,
By \eqref{eqnew0002} and \eqref{eqnew9987}, $\beta_{1,3}\|u_{1,\delta}u_{3,\delta}\|_2^2+\beta_{2,3}\|u_{2,\delta}u_{3,\delta}\|_2^2\geq0$ for sufficiently small $\delta>0$.  It contradicts \eqref{eqnew9991}.  Therefore, \eqref{eqnew0001} has no ground states for sufficiently small $\delta>0$.
\end{proof}

\begin{remark}\label{rmk0002}
By the proof of Proposition~\ref{prop0002}, we can obtain a by-product:  Suppose $\overrightarrow{u}_\delta$ is a ground state of \eqref{eqnew0001} for sufficiently small $\delta>0$, in the total-mixed case~$(d)$ with $\lambda_1<\min\{\lambda_2,\lambda_3\}$ and $s<\min\{1,t\}$.  Then, by \eqref{eqnew0002} and \eqref{eqnew9987},
\begin{eqnarray*}
\beta_{1,3}\|u_{1,\delta}u_{3,\delta}\|_2^2+\beta_{2,3}\|u_{2,\delta}u_{3,\delta}\|_2^2\geq0
\end{eqnarray*}
for sufficiently small $\delta>0$ in the case~$(i)$, which is given by Lemma~\ref{lem0002}.  It follows that
\begin{eqnarray*}
\left\{\aligned &C'\delta^te^{-2\sqrt{\lambda_1}|y_{2,\delta}-y_{3,\delta}|}-C\delta^se^{-2\min\{\sqrt{\lambda_2}, \sqrt{\lambda_3}\}|y_{2,\delta}-y_{3,\delta}|}\geq0,\quad\lambda_2\not=\lambda_3;\\
&C'\delta^te^{-2\sqrt{\lambda_1}|y_{2,\delta}-y_{3,\delta}|}-C\delta^s|y_{2,\delta}-y_{3,\delta}|^\alpha e^{-2\sqrt{\lambda_2}
|y_{2,\delta}-y_{3,\delta}|}\geq0,\quad\lambda_2=\lambda_3,\endaligned\right.
\end{eqnarray*}
which implies
$$
|y_{2,\delta}-y_{3,\delta}|\lesssim(\log\frac{1}{\delta})^{\frac{t-s}{2(\min\{\sqrt{\lambda_2}, \sqrt{\lambda_3}\}-\sqrt{\lambda_1})}}
$$
in the case~$(i)$.  Similarly, in the case~$(ii)$ which is given by Lemma~\ref{lem0002},
$$
|y_{2,\delta}-y_{3,\delta}|\lesssim(\log\frac{1}{\delta})^{\frac{1-s}{2(\min\{\sqrt{\lambda_2}, \sqrt{\lambda_3}\}-\sqrt{\lambda_1})}}.
$$
\end{remark}

We close this section by

\vskip0.12in

\noindent\textbf{Proof of Theorem~\ref{coro0001}:}\quad  The conclusion~$(1)$ follows from Propositions~\ref{prop0014} and \ref{prop0005} and \cite[Theorem~1]{LW05} (see also \cite[Corollary~1.3]{ST16}), the conclusion~$(2)$ follows from \cite[Theorem~3]{LW05} (see also \cite[Theorem~1.6]{ST16}), the conclusion~$(3)$ follows from Propositions~\ref{prop0001} and \ref{prop0003}, and the conclusion~$(4)$ follows from Proposition~\ref{prop0002}.
\hfill$\Box$

\section{$k$-coupled system~\eqref{eqn0001}}
In this section, we will consider  the general $k$-coupled system~\eqref{eqn0001} and prove Theorems~\ref{thm0003} and \ref{thm0002}.   Since the main ideas are similar to those of Theorem~\ref{coro0001}, we only sketch the proofs.

\medskip

\noindent\textbf{Proof of Theorem~\ref{thm0003}:}\quad  $(1)$
Since the proof of the existence of ground states of \eqref{eqnewnew0001} in the total-mixed case~$(H)$ with Morse index $4$ is similar to the Morse index 3 case of Theorem \ref{coro0001}, we shall only give the proof of the Morse index $3$ case.  Let
\begin{eqnarray*}
\mathcal{M}_{12,3,4}=\{\overrightarrow{u}\in\widehat{\mathcal{H}}_{12,3,4}\mid \overrightarrow{\mathcal{\widehat{Q}}}_{12,3,4}(u)=(\mathcal{G}_1(\overrightarrow{u})+\mathcal{G}_2(\overrightarrow{u}), \mathcal{G}_3(\overrightarrow{u}), \mathcal{G}_4(\overrightarrow{u}))=\overrightarrow{0}\},
\end{eqnarray*}
where $\mathcal{G}_j(\overrightarrow{u})=\|u_j\|_{\lambda_j}^2-\mu_j\|u_j\|_4^4-\sum_{i=1,i\not=j}^4\beta_{i,j}\|u_iu_j\|_2^2$ and $\widehat{\mathcal{H}}_{12,3,4}=((\mathcal{H}_1\times\mathcal{H}_2)\backslash\{\overrightarrow{0}\})\times(\mathcal{H}_3\backslash\{0\})
\times(\mathcal{H}_4\backslash\{0\})$.
Let
\begin{eqnarray*}
\mathcal{C}_{\mathcal{M}_{12,3,4}}=\inf_{\mathcal{M}_{12,3,4}}\mathcal{E}(\overrightarrow{u}).
\end{eqnarray*}
Then, $\mathcal{C}_{\mathcal{M}_{12,3,4}}$ is well defined and nonnegative.  Since $\beta_{1,2}>\widehat{\beta}_0>0$ and $\beta_{i,j}<\beta_{0}$ for all other $(i,j)\not=(1,2)$, where $\widehat{\beta}_0$ is sufficiently large and $\beta_{0}$ is sufficiently small, it is standard to show that $\mathcal{C}_{\mathcal{M}_{12,3,4}}<\sum_{j=1}^4\mathcal{E}_j(w_j)$.  Moreover, by similar arguments, as that used in the proof of Lemma~\ref{lemn0002}, we can show that
the matrix $\Xi=[\beta_{i,j}\|u_iu_j\|_2^2]_{i,j=1,2,\cdots,4}$ is strictly diagonally dominant for $\overrightarrow{u}\in\mathcal{N}_{\mathcal{M}_{12,3,4}}$ with $\sum_{j=1}^4\|u_j\|_{\lambda_j}^2\leq8\sum_{j=1}^4\mathcal{E}_j(w_j)$.  Here, $\beta_{j,j}=\mu_j$.  It follows that $\Xi$ is
positively definite with $|\text{det}(\Xi)|\geq C$.  Thus,
by similar arguments, as in  the proof of Lemma~\ref{lemn0002}, there exists a $(PS)$ sequence $\{\overrightarrow{u}_n\}$ at the least energy value $\mathcal{C}_{\mathcal{M}_{12,3,4}}$.  Moreover, any positive minimizer is a ground state of \eqref{eqnewnew0001} with the Morse index $3$.  Thus, it is sufficient to find a positive minimizer of $\mathcal{E}(\overrightarrow{u})$ on $\mathcal{M}_{12,3,4}$.  We start by estimating $\mathcal{C}_{\mathcal{M}_{12,3,4}}$.  By our assumptions, it is easy to verify that the degrees of eventual block decompositions of $\mathbf{A}_1$ all equal to $1$.  Thus, we can further group $\mathbf{A}_1$ which is given by \eqref{eqnewnew0007} into $\mathbf{A}_2$ which is given by \eqref{eqnewnew0008}.
Since the interaction force $\mathfrak{F}_{1,2}^0$, given by \eqref{eqnewnew0002}, is positive,  by Lemma~\ref{lem0003}, the least energy value of ground states in the block $C_{1,1}$, denoted by $\mathcal{C}_{\mathcal{M}_{12,3}}$, is strictly less than $\mathcal{C}_{\mathcal{N}_{1,2}}+\mathcal{E}_3(w_3)$.  Under the permutation: $(1,2,3,4)\to(1,2,4,3)$, there is another choice of $C_{1,1}$, which is consisted by $(u_1,u_2)$ and $u_4$.  Similarly, this least energy value of ground states, denoted by $\mathcal{C}_{\mathcal{M}_{12,4}}$, is also strictly less than $\mathcal{C}_{\mathcal{N}_{1,2}}+\mathcal{E}_4(w_4)$.  Thus, by \cite[Theorems~1 and 2]{AC06} and our choice that $\beta_{1,2}>\widehat{\beta}_0$ sufficiently large,
\begin{eqnarray*}
\mathcal{C}=\min\{\mathcal{C}_{\mathcal{M}_{12,3}}+\mathcal{E}_4(w_4), \mathcal{C}_{\mathcal{M}_{12,4}}+\mathcal{E}_3(w_3)\}
\end{eqnarray*}
is the smallest energy value that the $(PS)$ sequence, at the least energy value $\mathcal{C}_{\mathcal{M}_{12,3,4}}$, will split into blocks in passing to the limit in the optimal block decomposition~$\mathbf{A}_1$.  Even though there is another optimal block decomposition consisted by the blocks $(u_1,u_3)$, $u_2$ and $u_4$, by the assumptions $\beta_{1,2}>\widehat{\beta}_0>0$ and $\beta_{i,j}<\beta_{0}$ for all other $(i,j)\not=(1,2)$, the smallest energy value in this optimal block decomposition, defined similarly as $\mathcal{C}$, is strictly large than $\mathcal{C}$.  Thus, $\mathcal{C}$ is the smallest energy value that the $(PS)$ sequence, at the least energy value $\mathcal{C}_{\mathcal{M}_{12,3,4}}$, will split into blocks in passing to the limit.  Now, using the fact that the degrees of eventual block decompositions of $\mathbf{A}_1$ all equal to $1$ and similar arguments as that used in the proof of Lemma~\ref{lem0003} yields
$\mathcal{C}_{\mathcal{M}_{12,3,4}}<\mathcal{C}$.
Thus, applying the arguments similar to the proof of Proposition~\ref{prop0003} yields that $\mathcal{E}(\overrightarrow{u})$ has a positive minimizer on $\mathcal{M}_{12,3,4}$.

$(2)$\quad
Since we assume that all $|\beta_{i,j}|$ sufficiently small, the ground states, if they exist, should be minimizers of $\mathcal{E}(\overrightarrow{u})$ on
\begin{eqnarray*}
\mathcal{N}_{1,2,3,4}=\{\overrightarrow{u}\in\widehat{\mathcal{H}}_{1,2,3,4}\mid \overrightarrow{\mathcal{\widehat{Q}}}_{1,2,3,4}(u)=(\mathcal{G}_1(\overrightarrow{u}), \mathcal{G}_2(\overrightarrow{u}), \mathcal{G}_3(\overrightarrow{u}), \mathcal{G}_4(\overrightarrow{u}))=\overrightarrow{0}\},
\end{eqnarray*}
where $\mathcal{G}_j(\overrightarrow{u})=\|u_j\|_{\lambda_j}^2-\mu_j\|u_j\|_4^4-\sum_{i=1,i\not=j}^4\beta_{i,j}\|u_iu_j\|_2^2$ and $\widehat{\mathcal{H}}_{1,2,3,4}=((\mathcal{H}_1\backslash\{0\})\times(\mathcal{H}_2\backslash\{0\})\times(\mathcal{H}_3\backslash\{0\})
\times(\mathcal{H}_4\backslash\{0\})$.  Let
\begin{eqnarray*}
\mathcal{C}_{\mathcal{N}_{1,2,3,4}}=\inf_{\mathcal{N}_{1,2,3,4}}\mathcal{E}(\overrightarrow{u}).
\end{eqnarray*}
Then, by a similar choice of test functions as that in the proof of Lemma~\ref{lem0001},
\begin{eqnarray}\label{eqnewnew0006}
\mathcal{C}_{\mathcal{N}_{1,2,3,4}}\leq\mathcal{C}_{\mathcal{N}_{1,2}}+\mathcal{E}_{3}(w_3)+\mathcal{E}_4(w_4).
\end{eqnarray}
On the other hand, by similar arguments as used for \eqref{eqnew9987},
\begin{eqnarray}\label{eqnewnew0005}
\mathcal{C}_{\mathcal{N}_{1,2,3,4}}&\geq&\mathcal{C}_{\mathcal{N}_{1,2}}+\mathcal{E}_{3}(w_3)+\mathcal{E}_4(w_4)\notag\\
&&-\frac{1+o_\delta(1)}{2}(\beta_{1,3}\|u_1^\delta u_3^\delta\|_2^2+\beta_{2,3}\|u_2^\delta u_3^\delta\|_2^2+\beta_{1,4}\|u_1^\delta u_4^\delta\|_2^2\notag\\
&&+\beta_{2,4}\|u_2^\delta u_4^\delta\|_2^2+\beta_{3,4}\|u_3^\delta u_4^\delta\|_2^2).
\end{eqnarray}
Thus, since $t_{1,2}<\min\{t_{1,3}, t_{2,4}\}$, we can apply the arguments used in the proof of Lemma~\ref{lem0002} to show that $|y_{1,\delta}-y_{2,\delta}|\lesssim1$ and $|y_{i,\delta}-y_{i,\delta}|\to+\infty$ for $(i,j)\not=(1,2)$, where $y_{i,\delta}$ is the maximum point of $u_i^\delta$, respectively.  Moreover, similar computations as \eqref{eqnew9996} and \eqref{eqnew6666} yields $y_{1,\delta}-y_{2,\delta}\to0$ as $\delta\to0$.  Now, we can use Lemma~\ref{lemn0010} and similar computations as that in the proof of Proposition~\ref{prop0002} to estimate the term $\beta_{1,3}\|u_1^\delta u_3^\delta\|_2^2+\beta_{2,3}\|u_2^\delta u_3^\delta\|_2^2+\beta_{1,4}\|u_1^\delta u_4^\delta\|_2^2+\beta_{2,4}\|u_2^\delta u_4^\delta\|_2^2+\beta_{3,4}\|u_3^\delta u_4^\delta\|_2^2$.  Since $\min\{t_{2,3}, t_{1,4}, t_{3,4}\}<t_{1,2}$ and $\min\{\lambda_3,\lambda_4\}<\min\{\lambda_1,\lambda_2\}$,
$$
\beta_{1,3}\|u_1^\delta u_3^\delta\|_2^2+\beta_{2,3}\|u_2^\delta u_3^\delta\|_2^2+\beta_{1,4}\|u_1^\delta u_4^\delta\|_2^2+\beta_{2,4}\|u_2^\delta u_4^\delta\|_2^2+\beta_{3,4}\|u_3^\delta u_4^\delta\|_2^2<0
$$
for $\delta>0$ sufficiently small.  This  contradicts with \eqref{eqnewnew0006} and \eqref{eqnewnew0005}.  As a result, the ground states of \eqref{eqnewnew0001} do not exist.
\hfill$\Box$

\medskip

We close this section by

\medskip

\noindent\textbf{Proof of Theorem~\ref{thm0002}:}\quad  $(1)$  In proving this conclusion, we need to further employ the iteration argument.  We assume this conclusion is true for $3,4,\cdots,k-1$.  Recall that we have assumed that $\{i\not=j,(i,j)\in\mathcal{K}_{s,s,\mathbf{a}_d}\}\not=\emptyset$ for $s=1,2,\cdots,s_0$ and $\{i\not=j,(i,j)\in\mathcal{K}_{s,s,\mathbf{a}_d}\}=\emptyset$ for $s=s_0+1,\cdots,d$ with an $s_0\in\{0,1,2,\cdots,d\}$.
Since $d\leq\gamma\leq k$, there exists a unique $0\leq s^*\leq s_0$ such that $a_{s^*}\leq k-\gamma<a_{s^*+1}$.  Now,
we define the following Nihari manifold:
\begin{eqnarray*}
\mathcal{N}_{\gamma}=\bigg\{\overrightarrow{u}\in\widetilde{\mathcal{H}}_{\gamma}&\mid& \sum_{j=a_{s-1}+1}^{a_{s}}\mathcal{G}_{j}(\overrightarrow{u})=0, \quad \sum_{j=a_{s^*}+1}^{k-\gamma+1}\mathcal{G}_{j}(\overrightarrow{u})=0,\quad\mathcal{G}_{t}(\overrightarrow{u})=0\\
&&\mathcal{G}_{a_n}(u)=0,\quad
1\leq s\leq s^*, k-\gamma+2\leq t\leq a_{s_0}, s_0+1\leq n\leq m\bigg\},
\end{eqnarray*}
where $\mathcal{G}_j(\overrightarrow{u})=\|u_j\|_{\lambda_j}^2-\mu_j\|u_j\|_4^4-\sum_{i=1,i\not=j}^k\beta_{i,j}\|u_iu_j\|_2^2$,
$\mathcal{G}_{a_n}(u)=\|u\|_{\lambda_{a_n}}^2-\mu_{a_n}\|u\|_{4}^4$
and
$$
\widetilde{\mathcal{H}}_{\gamma}=\prod_{s=1}^{s^*}\bigg((\prod_{i=a_{s-1}+1}^{a_s}\mathcal{H}_{i})\backslash\{\overrightarrow{0}\}\bigg)
\times\bigg((\prod_{i=a_{s^*}+1}^{k-\gamma+1}\mathcal{H}_i)\backslash\{\overrightarrow{0}\}\bigg)
\times\bigg(\prod_{i=k-\gamma+2}^{k}(\mathcal{H}_{s}\backslash\{0\})\bigg).
$$
Let
\begin{eqnarray*}
\mathcal{C}_{\mathcal{N}_{\gamma}}=\inf_{\mathcal{N}_{\gamma}}\mathcal{E}(\overrightarrow{u}).
\end{eqnarray*}
Then $\mathcal{C}_{\mathcal{N}_{\gamma}}$ is nonnegative and well defined.  Since all $s_{th}$ inner-couplings are positive, it is standard to show that $\mathcal{C}_{\mathcal{N}_{\gamma}}\leq\sum_{j=1}^k\mathcal{E}_j(w_j)$.
Thus,
by similar arguments, as in the proof of Lemma~\ref{lemn0002} for $\gamma<k$ and also in the proof of Lemma~\ref{lemn0001} for $\gamma=k$, there exists a $(PS)$ sequence $\{\overrightarrow{u}_n\}$ at the least energy value $\mathcal{C}_{\mathcal{N}_{\gamma}}$.  Moreover, any positive minimizers of $\mathcal{E}(\overrightarrow{u})$ on $\mathcal{N}_{\gamma}$ is a ground state with the Morse index $\gamma$.  Thus, it is sufficient to show that there exists a positive minimizer of $\mathcal{E}(\overrightarrow{u})$ on $\mathcal{N}_{\gamma}$.  Recall that
\begin{eqnarray*}
\mathbf{A}_{d^\varsigma}^\varsigma=[\Theta_{t,s}^\varsigma]_{t,s=1,2,\cdots,d^\varsigma}
\end{eqnarray*}
be the $\varsigma_{th}$ decomposition.  Here,
\begin{eqnarray*}
\Theta_{t,s}^\varsigma=[\Theta_{i,j}^{\varsigma-1}]_{(i,j)\in\mathcal{K}_{t,s,\mathbf{a}^\varsigma_{d^\varsigma}}}
\end{eqnarray*}
and $0\leq\varsigma\leq\tau$,
\begin{eqnarray*}
\mathcal{K}_{t,s,\mathbf{a}^\varsigma_{d^\varsigma}}=(a^\varsigma_{t-1}, a^\varsigma_{t}]_{\bbn}\times(a^\varsigma_{s-1}, a^\varsigma_{s}]_{\bbn}
\end{eqnarray*}
with $\mathbf{a}^\varsigma_{d^\varsigma}=(a^\varsigma_0,a^\varsigma_1,\cdots,a^\varsigma_{d^\varsigma})$, $(a^\varsigma_{t-1}, a^\varsigma_{t}]_{\bbn}=(a^\varsigma_{t-1}, a^\varsigma_{t}]\cap\bbn$ and $0=a^\varsigma_0<a^\varsigma_1<\cdots<a^\varsigma_{d^\varsigma-1}<a^\varsigma_{d^\varsigma}=d^{\varsigma-1}$.  Since the eventual block decomposition $\mathbf{A}_{d^\tau}^\tau$ has the degree $m=1$, by the iteration assumptions, in every $\Theta_{s,s}^\varsigma$, there exists a ground state $\overrightarrow{u}_{s,\varsigma}$.  Moreover, by similar estimates as that in Lemma~\ref{lem0003}, the least energy value of $\overrightarrow{u}_{s,\varsigma}$ is strictly less than the sum of the least energy values of $\overrightarrow{u}_{i,\varsigma-1}$ for $i\in(a^\varsigma_{s-1}, a^\varsigma_{s}]_{\bbn}$.  Since all eventual block decompositions have the degree $m=1$, this fact also holds for all other eventual block decompositions.  Thus, in passing to a limit, if the $(PS)$ sequence $\{\overrightarrow{u}_n\}$ at the least energy value $\mathcal{C}_{\mathcal{N}_{\gamma}}$ will split into several blocks and some of them vanish at infinity, then the smallest energy value is generated by the sum of the least energy values of ground states, denoted by $\overrightarrow{u}_1^*$ and $\overrightarrow{u}_2^*$, in the $(s,s)$ blocks of the following decomposition
\begin{eqnarray*}
\widetilde{\mathbf{A}}=\left(\aligned C_{1,1}\quad C_{1,2}\\
C_{1,2}\quad C_{2,2}\endaligned\right),
\end{eqnarray*}
where $\widetilde{\mathbf{A}}$ is the last second decomposition of an optimal block decomposition.  Since all eventual block decompositions have the degree $m=1$, using $\overrightarrow{u}_1^*$ and $\overrightarrow{u}_2^*$ as basic elements to construct test functions as that in Lemma~\ref{lem0003} yields that $\mathcal{C}_{\mathcal{N}_{\gamma}}$ is strictly less than the sum of the least energy values of $\overrightarrow{u}_1^*$ and $\overrightarrow{u}_2^*$.  Thus, applying the Lions lemma and the Sobolev embedding theorem, similar as that in the proofs of Propositions~\ref{prop0001} and \ref{prop0003}, yields that the $(PS)$ sequence $\{\overrightarrow{u}_n\}$ at the least energy value $\mathcal{C}_{\mathcal{N}_{\gamma}}$ will not split such that some blocks vanish at infinity in passing to a limit.  It follows that there exists a minimizer of $\mathcal{E}(\overrightarrow{u})$ on $\mathcal{N}_{\gamma}$.  By the Harnack inequality, there exists a positive minimizer of $\mathcal{E}(\overrightarrow{u})$ on $\mathcal{N}_{\gamma}$.
In the purely attractive case, since the $\{\widehat{\rho}_{ij,l}\}$, given by \eqref{eqnew3344}, are nonincreasing for $k$, the existence of ground states in the purely attractive case can also be obtained by iteration the arguments of Propositions~\ref{prop0014} and \ref{prop0005} from $3$ to $k$, under the similar assumptions on $\lambda_j$ and $\beta_{i,j}$.

$(2)$\quad For $(2)$ of Theorem~\ref{thm0002}, as in the proof of Proposition~\ref{prop0002}, we still assume the contrary that, \eqref{eqn0001} has a ground state $\overrightarrow{u}_\delta$ under the assumptions of $(2)$ of Theorem~\ref{thm0002} for $\delta>0$ sufficiently small.
We define functionals as follows:
\begin{eqnarray*}\label{eqn0004}
\mathcal{E}_{s}(\overrightarrow{u})=\sum_{j=a_{s-1}+1}^{a_s}(\frac{1}{2}\|u_j\|_{\lambda_j}^2
-\frac{1}{4}\mu_j\|u_j\|_{4}^4)-\frac12\sum_{i\not=j,(i,j)\in\mathcal{K}_{s,s,\mathbf{a}_d}}\beta_{i,j}\|u_iu_j\|_2^2
\end{eqnarray*}
for $s=1,2,\cdots,s_0$ and $\mathcal{E}_{a_s}(u)=\frac{1}{2}\|u\|_{\lambda_{a_s}}^2-\frac{\mu_{a_s}}{4}\|u\|_{4}^4$
for $s=s_0+1,\cdots,m$.  We define the corresponding Nihari manifolds as follows:
\begin{eqnarray*}
\mathcal{N}_s=\{\overrightarrow{u}\in\prod_{j=a_{s-1}+1}^{a_s}(\mathcal{H}_j\backslash\{0\})\mid (\mathcal{G}_{a_{s-1}+1,s}(\overrightarrow{u}), \cdots, \mathcal{G}_{a_s,s}(\overrightarrow{u}))=\overrightarrow{0}\}
\end{eqnarray*}
with
\begin{eqnarray*}
\mathcal{G}_{j,s}(\overrightarrow{u})=\|u_j\|_{\lambda_j}^2-\mu_j\|u_j\|_{4}^4
-\sum_{i=a_{s-1}+1,i\not=j}^{a_s}\beta_{i,j}\|u_iu_j\|_2^2
\end{eqnarray*}
for $s=1,2,\cdots,s_0$ and
\begin{eqnarray}\label{eqn0034}
\mathcal{M}_{a_s}=\{\overrightarrow{u}\in\mathcal{H}_{a_s}\backslash\{0\}\mid \mathcal{Q}_{a_s}(u):=\|u\|_{\lambda_{a_s}}^2-\mu_{a_s}\|u\|_{4}^4=0\}
\end{eqnarray}
for $s=s_0+1,\cdots,m$.  Let
\begin{eqnarray*}
\mathcal{C}_{\mathcal{N}_{s}}=\inf_{\mathcal{N}_{s}}\mathcal{E}_{s}(\overrightarrow{u})
\quad\text{and}\quad
\mathcal{C}_{\mathcal{M}_{a_s}}=\inf_{\mathcal{M}_{a_s}}\mathcal{E}_{a_s}(u).
\end{eqnarray*}
Then $\mathcal{C}_{\mathcal{N}_{s}}$ and $\mathcal{C}_{\mathcal{M}_{a_s}}$ are all well defined and nonnegative.
As in Remark~\ref{rmk0001}, since $\beta_{i,j}>0$ in $\mathcal{E}_{l}(\overrightarrow{u})$ for all $(i,j)\in\{i\not=j,(i,j)\in\mathcal{K}_{l,l,\mathbf{a}_d}\}$ and all $1\leq l\leq s_0$,
\begin{eqnarray}\label{eqn0008}
\mathcal{C}_{\mathcal{N}_l}\leq\sum_{j=a_{l-1}+1}^{a_l}\mathcal{E}_j(w_j)-\frac{(1+o_\delta(1))}{2}
\sum_{(i,j)\in\mathcal{K}_{l,l,\mathbf{a}_d};i\not=j}\beta_{i,j}\|w_iw_j\|_2^2
\end{eqnarray}
for $1\leq l\leq s_0$.
On the other hand, by similar calculations as for \eqref{eqnew9997},
\begin{eqnarray}\label{eqn0009}
\mathcal{C}_{\mathcal{N}}\geq\sum_{j=1}^k\mathcal{E}_j(w_j)-\frac{(1+o_\delta(1))}{2}
\sum_{s,t=1}^{d}\sum_{(i,j)\in\mathcal{K}_{s,t,\mathbf{a}_d};i\not=j}\beta_{i,j}\|u_{i,\delta}u_{j,\delta}\|_2^2.
\end{eqnarray}
It follows from $t_{max,-}<t_{min,+}$ and $t_{0}< t_{min,int,+}$ that
\begin{eqnarray}\label{eqn0012}
\|u_{i,\delta}u_{j,\delta}\|_2^2=o_\delta(1)\quad\text{for all }(i,j)\in\mathcal{K}_{t,s,\mathbf{a}_m}, i\not=j\quad\text{and}\quad t\not=s.
\end{eqnarray}
By Lions' lemma and the Sobolev embedding theorem,  there exists $\{y_{j,\delta}\}\subset\bbr^N$ such that $u_{j,\delta}(x+y_{j,\delta})\to w_j$ strongly in $\mathcal{H}_j$ as $\delta\to0$ up to a subsequence.
Applying the Moser iteration and the elliptic estimates, as that used in the proof of Lemma~\ref{lem0002}, yields that $u_{j,\delta}(x+y_{j,\delta})\to w_j$ strongly in $L^\infty(\bbr^N)$ as $\delta\to0$ up to a subsequence.  Without loss of generality, $y_{j,\delta}$ can be chosen to be the maximum point of $u_{j,\delta}$.  By a similar argument as for \eqref{eqnew0002}, it is standard to show that
\begin{eqnarray*}
\mathcal{C}_{\mathcal{N}}\leq\sum_{l=1}^{s_0}\mathcal{C}_{\mathcal{N}_{l}}+\sum_{s=s_0+1}^d\mathcal{C}_{\mathcal{M}_{a_s}}.
\end{eqnarray*}
Thus, by a similar calculation as for \eqref{eqnew9987},
\begin{eqnarray}\label{eqnew3355}
\sum_{s,t=1;s<t}^{d}\sum_{(i,j)\in\mathcal{K}_{s,t,\mathbf{a}_m}}\beta_{i,j}\|u_{i,\delta}u_{j,\delta}\|_2^2\geq0
\end{eqnarray}
for $\delta>0$ sufficiently small.  It follows from \eqref{eqn0008} and \eqref{eqn0009} that
\begin{eqnarray}\label{eqn0011}
\sum_{l=1}^{s_0}\sum_{(i,j)\in\mathcal{K}_{l,l,\mathbf{a}_d};i\not=j}\widehat{\beta}_{i,j}\|u_{i,\delta}u_{j,\delta}\|_2^2\geq
\sum_{l=1}^{s_0}\sum_{(i,j)\in\mathcal{K}_{l,l,\mathbf{a}_d};i\not=j}\widehat{\beta}_{i,j}\|w_iw_{j}\|_2^2+o_{\delta}(1).
\end{eqnarray}
Thus, $y_{i,\delta}-y_{j,\delta}=y_{ij}+o_\delta(1)$ and
\begin{eqnarray}\label{eqn0010}
1\lesssim\|u_{i,\delta}u_{j,\delta}\|_2^2\quad\text{for all }(i,j)\in\mathcal{K}_{l,l,\mathbf{a}_d}, i\not=j\text{ and all }l=1,2,\cdots,s_0
\end{eqnarray}
Let $F(\mathbf{y})=\sum_{s=1}^{d}\sum_{(i,j)\in\mathcal{K}_{s,s,\mathbf{a}_d};i\not=j}\beta_{i,j}\|w_iw_{j,y_{ij}}\|_2^2$.  Since $w_j(x)$ is strictly decreasing for $|x|$, by a similar argument as that used for \eqref{eqnew6666}, $\nabla F(\mathbf{y})=0$ if and only if $\mathbf{y}=\mathbf{0}$.  Thus, by \eqref{eqn0011}, $y_{ij}=0$ for all $(i,j)\in\mathcal{K}_{l,l,\mathbf{a}_d}$ with $i\not=j$ and all $l=1,2,\cdots,s_0$.  Thus, without loss of generality, we assume $y_{i,\delta}=y_{j,\delta}=y_{l,\delta}$ for all $(i,j)\in\mathcal{K}_{l,l,\mathbf{a}_d}$ with $i\not=j$ and all $l=1,2,\cdots,s_0$ with $\delta>0$ sufficiently small.  By \eqref{eqn0012} and the Lebesgue dominated convergence theorem, $|y_{l,\delta}-y_{l',\delta}|\to+\infty$ for all $l,l'=1,2,\cdots,d$ with $l\not=l'$.  We denote $y_{l,\delta}-y_{l',\delta}$ by $y_{ll',\delta}$, for the sake of simplicity.  Then, by similar arguments as for \eqref{eqnew9994} and \eqref{eqnew9993},
\begin{eqnarray}
&&\sum_{s,t=1;s<t}^{d}\sum_{(i,j)\in\mathcal{K}_{s,t,\mathbf{a}_d}}\beta_{i,j}\|u_{i,\delta}u_{j,\delta}\|_2^2\notag\\
&=&\sum_{s,t=1;s<t}^{d}\sum_{(i,j)\in\mathcal{K}_{s,t,\mathbf{a}_d}}\bigg(\sum_{\lambda_i=\lambda_j}
C_{i,j}\beta_{i,j}(\frac{1}{|y_{st,\delta}|})^{N-1-\alpha}e^{-2\sqrt{\lambda_i}|y_{st,\delta}|}\notag\\
&&+\sum_{\lambda_i\not=\lambda_j}C_{i,j}\beta_{i,j}(\frac{1}{|y_{st,\delta}|})^{N-1}
e^{-2\min\{\sqrt{\lambda_i},\sqrt{\lambda_j}\}|y_{st,\delta}|}\bigg).
\label{eqn0013}
\end{eqnarray}
Since
\begin{eqnarray*}
\min\{\sqrt{\lambda_{i_0}}, \sqrt{\lambda_{j_0}}\}\geq\min\{\sqrt{\lambda_{i_0'}}, \sqrt{\lambda_{j_0'}}\}
\end{eqnarray*}
for all $(i_0,j_0)$ and $(i_0',j_0')$ with $\beta_{i_0,j_0}>0>\beta_{i_0',j_0'}$, by $t_{max,-}<t_{min,+}$ and \eqref{eqn0013},
\begin{eqnarray*}
\sum_{s,t=1;s<t}^{d}\sum_{(i,j)\in\mathcal{K}_{s,t,\mathbf{a}_d}}\beta_{i,j}\|u_{i,\delta}u_{j,\delta}\|_2^2<0
\end{eqnarray*}
for $\delta>0$ sufficiently small, which contradicts \eqref{eqnew3355}.  Hence, \eqref{eqn0001} has no ground states for $\delta>0$ sufficiently small under the conditions of $(2)$ of Theorem~\ref{thm0002}.

$(3)$\quad  In the purely repulsive case, this result has been proved in \cite{LW05}.  For the repulsive-mixed case, by regarding the blocks in optimal block decompositions as a whole, we can follow the argument as used in the proof of \cite[Theorem~3]{LW05} to show that the ground states of \eqref{eqnewnew0001} do not exist.
\hfill$\Box$

\section{Appendix:Proof of Lemma \ref{lemn0010}}
\begin{proof}
When $\lambda_i\not=\lambda_j$, the Lemma is proved in \cite[Lemma~6]{LW05}.  Thus, we assume that $\lambda_i=\lambda_j=\lambda$.
Let $M>0$ be sufficiently large but fixed such that the decay estimate~\eqref{eqnew9998} holds for $w_j$ with $|x|>M$.
We first consider the case $N=1$.  Without loss of generality, we assume that $\lambda_i=\lambda_j=\lambda_1$ and $w_i=w_j=w_1$.  Moreover, we also assume that $e_1=1$.  Then, $Re_1=R$ and for $R>0$ sufficiently large,
\begin{eqnarray*}
&&\int_{-\infty}^{+\infty}w_1^2(|x|)w_1^2(|x-R|)dx\\
&=&\int_{-\infty}^{-M}w_1^2(|x|)w_1^2(|x-R|)dx+\int_{-M}^{M}w_1^2(|x|)w_1^2(|x-R|)dx\\
&&+\int_{M}^{R-M}w_1^2(|x|)w_1^2(|x-R|)dx+\int_{R-M}^{R+M}w_1^2(|x|)w_1^2(|x-R|)dx\\
&&+\int_{R+M}^{+\infty}w_1^2(|x|)w_1^2(|x-R|)dx.
\end{eqnarray*}
By symmetry,
\begin{eqnarray*}
\int_{-\infty}^{-M}w_1^2(|x|)w_1^2(|x-R|)dx=\int_{R+M}^{+\infty}w_1^2(|x|)w_1^2(|x-R|)dx
\end{eqnarray*}
and
\begin{eqnarray*}
\int_{-M}^{M}w_1^2(|x|)w_1^2(|x-R|)dx=\int_{R-M}^{R+M}w_1^2(|x|)w_1^2(|x-R|)dx.
\end{eqnarray*}
For $\int_{-M}^{M}w_1^2(|x|)w_1^2(|x-R|)dx$, we estimate by \eqref{eqnew9998} as follows:
\begin{eqnarray*}
\int_{-M}^{M}w_1^2(|x|)w_1^2(|x-R|)dx&\sim&\int_{-M}^{M}w_1^2(|x|)e^{-2\sqrt{\lambda_1}|x-R|}dx\\
&=&\int_{-M}^{M}w_1^2(|x|)e^{-2\sqrt{\lambda_1}(R-x)}dx\\
&=&e^{-2\sqrt{\lambda_1}R}\int_{-M}^{M}w_1^2(|x|)e^{2\sqrt{\lambda_1}x}dx\\
&\sim&e^{-2\sqrt{\lambda_1}R}
\end{eqnarray*}
as $R\to+\infty$.
For $\int_{-\infty}^{-M}w_1^2(|x|)w_1^2(|x-R|)dx$, we estimate by \eqref{eqnew9998} as follows:
\begin{eqnarray*}
\int_{-\infty}^{-M}w_1^2(|x|)w_1^2(|x-R|)dx&\sim&\int_{-\infty}^{-M}e^{-2\sqrt{\lambda_1}|x|}e^{-2\sqrt{\lambda_1}|x-R|}dx\\
&=&\int_{-\infty}^{-M}e^{2\sqrt{\lambda_1}x}e^{-2\sqrt{\lambda_1}(R-x)}dx\\
&=&e^{-2\sqrt{\lambda_1}R}\int_{-\infty}^{-M}e^{4\sqrt{\lambda_1}x}dx\\
&\sim&e^{-2\sqrt{\lambda_1}R}
\end{eqnarray*}
as $R\to+\infty$.  For $\int_{M}^{R-M}w_1^2(|x|)w_1^2(|x-R|)dx$, we estimate by \eqref{eqnew9998} as follows:
\begin{eqnarray*}
\int_{M}^{R-M}w_1^2(|x|)w_1^2(|x-R|)dx&\sim&\int_{M}^{R-M}e^{-2\sqrt{\lambda_1}|x|}e^{-2\sqrt{\lambda_1}|x-R|}dx\\
&=&\int_{M}^{R-M}e^{-2\sqrt{\lambda_1}x}e^{-2\sqrt{\lambda_1}(R-x)}dx\\
&=&e^{-2\sqrt{\lambda_1}R}(R-2M)\\
&\sim& Re^{-2\sqrt{\lambda_1}R}
\end{eqnarray*}
as $R\to+\infty$.  Thus, $\int_{-\infty}^{+\infty}w_1^2(|x|)w_1^2(|x-R|)dx\sim Re^{-2\sqrt{\lambda_1}R}$ as $R\to+\infty$.  Without loss of generality, we assume that $e_1=(0,1)$ for $N=2$ and $e_1=(0,0,1)$ for $N=3$.  Thus, for the cases $N=2,3$, by symmetry,
\begin{eqnarray*}
&&\int_{\bbr^N}w_1^2(|x|)w_1^2(|x-Re_1|)dx\\
&=&\int_{\{|x|\leq M\}}w_1^2(|x|)w_1^2(|x-Re_1|)dx+\int_{\{|x-Re_1|\leq M\}}w_1^2(|x|)w_1^2(|x-Re_1|)dx\\
&&+\int_{\{M<|x|\leq \frac{R}{2}\}}w_1^2(|x|)w_1^2(|x-Re_1|)dx+\int_{\{M<|x-Re_1|\leq \frac{R}{2}\}}w_1^2(|x|)w_1^2(|x-Re_1|)dx\\
&&+\int_{\{|x|>\frac{R}{2}\}\cap\{|x-Re_1|>\frac{R}{2}\}}w_1^2(|x|)w_1^2(|x-Re_1|)dx\\
&=&2\int_{\{|x|\leq M\}}w_1^2(|x|)w_1^2(|x-Re_1|)dx+2\int_{\{M<|x|\leq \frac{R}{2}\}}w_1^2(|x|)w_1^2(|x-Re_1|)dx\\
&&+\int_{\{|x|>\frac{R}{2}\}\cap\{|x-Re_1|>\frac{R}{2}\}}w_1^2(|x|)w_1^2(|x-Re_1|)dx
\end{eqnarray*}
for $R>0$ sufficiently large.  For $\int_{\{|x|\leq M\}}w_1^2(|x|)w_1^2(|x-Re_1|)dx$, we estimate by \eqref{eqnew9998} as follows:
\begin{eqnarray*}
\int_{\{|x|\leq M\}}w_1^2(|x|)w_1^2(|x-Re_1|)dx&\sim&\int_{\{|x|\leq M\}}w_1^2(|x|)|x-Re_1|^{1-N}e^{-2\sqrt{\lambda_1}|x-Re_1|}dx\\
&\lesssim&R^{1-N}e^{-2\sqrt{\lambda_1}R}\int_{\{|x|\leq M\}}w_1^2(|x|)e^{2\sqrt{\lambda_1}|x|}dx
\end{eqnarray*}
as $R\to+\infty$.  For $\int_{\{|x|>\frac{R}{2}\}\cap\{|x-Re_1|>\frac{R}{2}\}}w_1^2(|x|)w_1^2(|x-Re_1|)dx$, we estimate by \eqref{eqnew9998} as follows:
\begin{eqnarray*}
&&\int_{\{|x|>\frac{R}{2}\}\cap\{|x-Re_1|>\frac{R}{2}\}}w_1^2(|x|)w_1^2(|x-Re_1|)dx\\
&\sim&\int_{\{|x|>\frac{R}{2}\}\cap\{|x-Re_1|>\frac{R}{2}\}}(|x||x-Re_1|)^{1-N}e^{-2\sqrt{\lambda_1}|x|}e^{-2\sqrt{\lambda_1}|x-Re_1|}dx\\
&\lesssim&R^{1-N}e^{-\sqrt{\lambda_1}R}\int_{\{|x|>\frac{R}{2}\}}|x|^{1-N}e^{-2\sqrt{\lambda_1}|x|}dx\\
&=&R^{1-N}e^{-\sqrt{\lambda_1}R}\int_{\frac{R}{2}}^{+\infty}e^{-2\sqrt{\lambda_1}r}dr\\
&\sim&R^{1-N}e^{-2\sqrt{\lambda_1}R}
\end{eqnarray*}
as $R\to+\infty$.  For $\int_{\{M<|x|\leq \frac{R}{2}\}}w_1^2(|x|)w_1^2(|x-Re_1|)dx$, we denote $x=(x',x_1)$.  Then,
\begin{eqnarray}\label{eqnew9654}
|Re_1-x|-R\sim-x_1+\frac{|x|^2}{2R}\quad\text{uniformly for }M<|x|\leq \frac{R}{2}.
\end{eqnarray}
Thus, by \eqref{eqnew9998} and $\frac{R}{2}\leq|x-Re_1|\leq\frac{3R}{2}$ uniformly for $M<|x|\leq \frac{R}{2}$,
\begin{eqnarray}
&&\int_{\{M<|x|\leq \frac{R}{2}\}}w_1^2(|x|)w_1^2(|x-Re_1|)dx\notag\\
&\sim&\int_{\{M<|x|\leq \frac{R}{2}\}}(|x||x-Re_1|)^{1-N}e^{-2\sqrt{\lambda_1}|x|}e^{-2\sqrt{\lambda_1}|x-Re_1|}dx\notag\\
&\sim&R^{1-N}e^{-2\sqrt{\lambda_1}R}\int_{\{M<|x|\leq \frac{R}{2}\}}|x|^{1-N}e^{-2\sqrt{\lambda_1}(|x|+\frac{|x|^2}{2R}-x_1)}dx\label{eqnew9964}
\end{eqnarray}
We estimate the upper bound as follows:
\begin{eqnarray*}
&&\int_{\{M<|x|\leq \frac{R}{2}\}}w_1^2(|x|)w_1^2(|x-Re_1|)dx\\
&\lesssim&R^{1-N}e^{-2\sqrt{\lambda_1}R}\int_{0}^{\frac{\pi}{2}}\int_{M}^{\frac{R}{2}}e^{-2\sqrt{\lambda_1}(r-rcos\rho)}drd\rho\\
&\lesssim&R^{1-N}e^{-2\sqrt{\lambda_1}R}\int_{0}^{\frac{\pi}{2}}\int_{M}^{\frac{R}{2}}e^{-2\sqrt{\lambda_1}r(\sin\rho)^2}drd\rho\\
&=&R^{1-N}e^{-2\sqrt{\lambda_1}R}\int_{0}^{\frac{\pi}{2}}\int_{M}^{\frac{R}{2}}e^{-2\sqrt{\lambda_1}r\rho^2(\frac{\sin\rho}{\rho})^2}drd\rho\\
&\sim&R^{1-N}e^{-2\sqrt{\lambda_1}R}\int_{0}^{\frac{\pi}{2}}\int_{M}^{\frac{R}{2}}e^{-2\sqrt{\lambda_1}r\rho^2}drd\rho\\
&\sim&R^{1-N}e^{-2\sqrt{\lambda_1}R}\int_{M}^{\frac{R}{2}}r^{-\frac{1}{2}}dr\int_{0}^{+\infty}e^{-2\sqrt{\lambda_1}y^2}dy\\
&\sim&R^{\frac{3}{2}-N}e^{-2\sqrt{\lambda_1}R}.
\end{eqnarray*}
For the lower bound, we estimate it as follows:
\begin{eqnarray*}
&&\int_{\{M<|x|\leq \frac{R}{2}\}}w_1^2(|x|)w_1^2(|x-Re_1|)dx\notag\\
&\gtrsim&R^{1-N}e^{-2\sqrt{\lambda_1}R}\int_{\{M<|x|\leq \frac{R}{2}\}}|x|^{1-N}e^{-2\sqrt{\lambda_1}(|x|+\frac{|x|^2}{2R}-x_1)}dx\\
&\gtrsim&R^{1-N}e^{-2\sqrt{\lambda_1}R}\int_{0}^{\frac{\pi}{4}}(\sin\rho)^{N-2}\int_{M}^{\frac{R}{2}}e^{-4\sqrt{\lambda_1}r\cos^2\rho}drd\rho\\
&\gtrsim&R^{1-N}e^{-2\sqrt{\lambda_1}R}\int_{0}^{\frac{\pi}{4}}\sin\rho\int_{M}^{\frac{R}{2}}e^{-4\sqrt{\lambda_1}r\cos^2\rho}drd\rho\\
&\sim&R^{1-N}e^{-2\sqrt{\lambda_1}R}\int_{M}^{\frac{R}{2}}r^{-\frac{1}{2}}dr\int_{0}^{+\infty}e^{-2\sqrt{\lambda_1}y^2}dy\\
&\sim&R^{\frac{3}{2}-N}e^{-2\sqrt{\lambda_1}R}.
\end{eqnarray*}
The proof is thus completed.
\end{proof}

\section{Acknowledgements}
The research of J. Wei is
partially supported by NSERC of Canada.  The research of Y. Wu is supported by NSFC (No. 11701554, No. 11771319), the Fundamental Research Funds for the Central Universities (2017XKQY091) and Jiangsu overseas visiting scholar program for university prominent young $\&$ middle-aged teachers and presidents.  This paper was completed when Y. Wu was visiting University of British Columbia.  He is
grateful to the members in Department of Mathematics at University of British Columbia for their invitation and hospitality.

\end{document}